\documentclass{article} 
\usepackage[T1]{fontenc}
\usepackage{amsmath}
\usepackage{amssymb}
\usepackage{amsthm}
\usepackage{cancel}
\usepackage{dsfont}
\usepackage{stmaryrd}
\SetSymbolFont{stmry}{bold}{U}{stmry}{m}{n}
\usepackage{hyperref}
\usepackage{empheq}
\usepackage{cleveref}
\usepackage{bm}
\usepackage{enumitem}
\usepackage[dvipsnames]{xcolor}
\usepackage[pdftex]{graphicx}
\usepackage{pgf, tikz}
\usepackage{pgfplots}
\usetikzlibrary{plotmarks}
\usetikzlibrary{patterns}
\usetikzlibrary{shadings}
\usepackage{fullpage}
\usepackage{color}

\title{Monge-Amp\`ere gravitation as a $\Gamma$-limit of good rate functions}
\author{Luigi \textsc{Ambrosio}\footnote{Scuola Normale Superiore, Pisa.
E-mail: \emph{ambrosio@sns.it}} \and Aymeric \textsc{Baradat}\footnote{Max Planck Institute for Mathematics in the Sciences, Leipzig.
E-mail: \emph{baradat@mis.mpg.de}} \and Yann \textsc{Brenier}\footnote{\'Ecole Normale Sup\'erieure, Paris. 
E-mail: \emph{brenier@dma.ens.fr}}}
\date{\today}

\theoremstyle{plain}
\newtheorem{Thm}{Theorem}

\newtheorem{Prop}[Thm]{Proposition}
\newtheorem{Lem}[Thm]{Lemma}
\newtheorem{Ass}[Thm]{Assumptions}
\theoremstyle{definition}
\newtheorem{Def}[Thm]{Definition}

\theoremstyle{remark}
\newtheorem{Rem}[Thm]{Remark}

\newcommand{\R}{\mathbb{R}}

\newcommand{\T}{\mathbb{T}}

\newcommand{\N}{\mathbb{N}}
\newcommand{\eps}{\varepsilon}
\newcommand{\perm}{\mathfrak{S}_N}
\newcommand{\XX}{\mathcal{X}}
\newcommand{\YY}{\mathcal{Y}}
\newcommand{\ZZ}{\mathcal{Z}}

\newcommand{\cg}{\langle}
\newcommand{\cd}{\rangle}

\DeclareMathOperator{\Div}{div}

\DeclareMathOperator{\tr}{tr}

\DeclareMathOperator{\D}{d \!}
\DeclareMathOperator{\DD}{D^2 \!}

\begin{document} 

\maketitle

\begin{abstract}
Monge-Amp\`ere gravitation is a modification of the classical Newtonian gravitation where the linear Poisson equation is replaced by the nonlinear Monge-Amp\`ere equation. This paper is concerned with the rigorous derivation of Monge-Amp\`ere gravitation for a finite number of particles from the stochastic model of a Brownian point cloud, in the spirit of the formal paper~\cite{brenier2015double}. The main step in this derivation is the $\Gamma-$convergence of the good rate functions corresponding to a one-parameter family of large deviation principles. Surprisingly, the derived model includes dissipative phenomena. As an illustration, we show that it leads to sticky collisions in one space dimension.
\end{abstract}

\setcounter{tocdepth}{1}
\tableofcontents

\section{Introduction}
On a periodic domain such as $\mathbb{T}^d=\mathbb{(R/Z)}^d$,
Newtonian gravitation is commonly
described  in terms of the density of probability $f(t,x,\xi)$ to find gravitating matter at time $t$, position 
$x\in \mathbb{T}^d$ and velocity $\xi\in\mathbb{R}^d$, subject to the Vlasov-Poisson equation
\begin{gather*}
\partial_t f(t,x,\xi)+\Div_x(\xi f(t,x,\xi))-\Div_{\xi}(\nabla \varphi(t,x)f(t,x,\xi))=0,\\
\Delta \varphi(t,x)=\int_{\mathbb{R}^d} f(t,x,\xi) \D \xi-1, \quad (t,x,\xi)\in\mathbb{R}\times \mathbb{T}^d\times\mathbb{R}^d,
\end{gather*}
where $\varphi$ is the gravitational potential. Notice that the averaged density, say $1$, has been subtracted out from the right-hand side of the Poisson equation, due to the periodicity of the spatial domain. This is a common feature of computational cosmology and it let the uniform density be a stationary solution. The Vlasov-Poisson system can be seen as an "approximation" to the more nonlinear 
Vlasov-Monge-Amp\`ere (VMA) system
\begin{gather}
\label{vlasov} \partial_t f(t,x,\xi)+\Div_x(\xi f(t,x,\xi))-\Div_{\xi}(\nabla\varphi(t,x)f(t,x,\xi))=0,\\
\label{monge-ampère} \det(\mathbb{I}+\DD\varphi(t,x))=\int_{\mathbb{R}^d} f(t,x,\xi) \D\xi
,\quad
(t,x,\xi)\in\mathbb{R}\times \mathbb{T}^d\times\mathbb{R}^d,
\end{gather}
where the fully nonlinear Monge-Amp\`ere equation substitutes for the linear Poisson equation of
Newtonian gravitation. Indeed, for "weak" gravitational potential, by expanding the determinant
about the identity matrix $\mathbb{I}$, we
get 
\[
\det(\mathbb{I}+\DD\varphi(t,x))\sim 1+ \tr(\DD\varphi(t,x))
=1+\Delta\varphi(t,x)
\]
and recover the Newtonian model approximately (and exactly as $d=1$).
In this paper, we will speak of "Monge-Amp\`ere gravitation" ("MAG" in short). The Vlasov-Monge-Amp\`ere system has been introduced and related to the Vlasov-Poisson system in~\cite{brenier2004geometric}, and studied as an ODE on the Wasserstein space in \cite{ambrosio2008hamiltonian}. 
It can also be solved numerically thanks to efficient Monge-Amp\`ere solvers recently designed by M\'erigot \cite{merigot2011multiscale}. It has been argued in \cite{brenier2011modified} that the MAG may also be seen as an approximation of Newtonian gravitation for which the "Zeldovich approximation"~\cite{zel1970gravitational} (see \cite{frisch2002reconstruction,brenier2003reconstruction}), popular in computational cosmology, becomes exact.

In this paper we will not be directly interested in this system, but rather in its discrete version, \textit{i.e.} when the number of particles is finite. As well known in optimal transport theory~\cite{brenier1987decomposition,brenier1991polar,villani2003topics}, the Monge-Amp\`ere equation~\eqref{monge-ampère} is solved by the unique function $\varphi$ such that the map $\mathrm{Id} + \nabla \varphi$ realizes the optimal transport with quadratic cost from the density $\int f \D \xi$ to the Lebesgue measure. Then, the kinetic equation~\eqref{vlasov} is known to be the continuous version of Newton equations of classical mechanics in a potential given by $\varphi$.

In the discrete setting, the stationary Lebesgue measure is replaced by a family $(a_1, \dots, a_N) \in (\R^{d})^N$ of $N \geq 1$ points in $\R^d$ (here we make the presentation in $\R^d$ instead of $\T^d$ for the sake of simplicity). One can for instance think of a regular lattice approximating in some region a constant density, even though in the sequel the particular choice of $(a_1,\dots, a_n)$ will play no role. We will consider the evolution of a cloud of $N$ particles $(x_1,\dots, x_n)$ in $\R^d$ whose dynamic is ruled by the discrete optimal transport problem:
\begin{equation}
\label{eq:discrete_OT}
\sigma_{\mathrm{opt}} = \sigma_{\mathrm{opt}}(X) := \underset{\sigma \in \mathfrak{S}_n}{\mathrm{Arginf}}\sum_{i=1}^N |x_i - a_{\sigma(i)}|^2, \quad X=(x_1, \dots, x_N) \in (\R^d)^N.
\end{equation}
More precisely, the analogous of \eqref{vlasov}\eqref{monge-ampère} in this framework is easily seen to be formally:\footnote{\label{footnote:ill-posed} Due to the lack of uniqueness in the discrete optimal transport problem, this system is not always well defined \emph{a priori} but we don't want to enter into the details here.}
\begin{equation}
\label{eq:formal_EDO}
\forall i = 1, \dots, N,\quad \frac{\D\,^2}{\D t^2} x_i(t) = x_i(t) - a_{\sigma_{\mathrm{opt}}(i)}.
\end{equation}

Following the idea of the recent paper \cite{brenier2015double}, we will derive this discrete dynamic from the very elementary stochastic model of a Brownian point cloud. However, in \cite{brenier2015double}, the derivation was obtained through a double application of the
large deviation principle (LDP), through a purely formal use of the Freidlin-Wenzell theory~\cite{freidlin1998random}. The main purpose of the present paper is to explain how such a derivation can be made rigorous by substituting for one of the applications of the LDP a PDE method inspired by the famous concept of "onde pilote" introduced by Louis de Broglie at the early stage of Quantum Mechanics~\cite{deBroglie1927mecanique}. 

The outline of the paper is the following. In Section \ref{sec:derivation} we show how to derive MAG starting from a finite number of Brownian particles. This will be done in several steps and we do not want to enter into the details now, but a key argument will be the $\Gamma-$convergence of the good rate functions associated with a family of SDEs towards an "effective" functional related to MAG. This is stated in Theorem \ref{mainresult}, which is our main result. Section \ref{sec:proof} is dedicated to the proof of Theorem \ref{mainresult}. The effective functional that we obtain does not lead \emph{exactly} to MAG as stated in~\eqref{eq:formal_EDO} (which as already saw in footnote \ref{footnote:ill-posed} is not well-posed in general), but also includes dissipative phenomena in those points $X$ where the solution of the discrete optimal transport problem~\eqref{eq:discrete_OT} is not unique. Even if we do not know for the moment how to treat these dissipative effects in general, the purpose of Section~\ref{sec:sticky} is to show that in $1$ space dimension, they lead to sticky collisions.

\underline{Notations}. We will work with $N$ particles in $\R^d$, and hence in $(\R^d)^N$. Points of $(\R^d)^N$ will be denoted with capital letters, mainly $X$, $Y$ or $Z$. Curves with values in $(\R^d)^N$ will be denoted with calligraphic letters $\XX$, $\YY$ or $\ZZ$. The position of $\XX$, $\YY$ and $\ZZ$ at time $t \in \R$ will be denoted by $X_t$, $Y_t$ and $Z_t$ respectively.

\section{Derivation of the discrete model}
\label{sec:derivation}
\subsection{The stochastic model of a lattice with Brownian agitation}
Take $A = (a_1, \dots, a_N) \in (\R^{d})^N$ a family of $N \geq 1$ points in $\R^d$. We assume each point of this lattice to be subject to Brownian agitation for times $t\geq 0$. At time $t$, the position of point $i$ is
\[
a_i + \sqrt{\eps} B_t^i,
\]
where $(B^i)_{i = 1, \dots, d}$ is a family of $N$ independent normalized Brownian curves and $\eps$ monitors the (common) level of noise. As a consequence, at time $t>0$, the density of probability $\rho_{\eps}(t,X)$ for the point cloud 
\[
(a_1 + \sqrt{\eps} B_t^1, \dots, a_N + \sqrt{\eps} B_t^N)
\]
to be observed at location $X = (x_1, \dots, x_d) \in (\R^d)^N$, up to a permutation $\sigma\in\mathfrak{S}_N$ of the labels, is easy to compute. We find
\[
\rho_{\eps}(t,X)=\frac{1}{N!\sqrt{2\pi\eps t}^{dN}} \sum_{\sigma\in\mathfrak{S}_N} \prod_{\alpha=1}^N \exp\left(-\frac{|x_i-a_{\sigma(i)}|^2}{2\eps t}\right),
\]
or, in short,
\[
\frac{1}{N!\sqrt{2\pi\eps t}^{Nd}} \sum_{\sigma\in\mathfrak{S}_N} \exp\left(-\frac{|X-A^\sigma|^2}{2\eps t}\right),
\]
where $|\cdot|$ denotes the euclidean norm in $\mathbb{R}^d$ or $(\R^d)^N$ depending on the context, and where for all $X = (x_1, \dots, x_N) \in (\R^d)^N$, $X^\sigma$ stands for:
\[
X^\sigma=(x_{\sigma(1)}, \dots, x_{\sigma(N)}).
\]
This was the starting point of the discussion made in \cite{brenier2015double}, using a double large deviation principle.

In the present paper, we rather turn to a PDE viewpoint, where $\rho_{\eps}$ is the solution of the heat equation in $(\R^d)^N$
\begin{equation}
\label{eq:heat_equation}
\frac{\partial\rho_{\eps}}{\partial t}(t,X)=\frac{\eps}{2}\Delta \rho_{\eps}(t,X)
\end{equation}
with, as initial condition, the delta measure located at
$A=(a_1, \dots, a_N) \in (\R^d)^N$ and symmetrized with respect to $\sigma\in\mathfrak{S}_N$, namely:
\begin{equation}
\label{eq:sym_initial_value}
\rho_{\eps}(0,X)=\frac{1}{N!}\sum_{\sigma\in\mathfrak{S}_N} \delta_{A^\sigma}.
\end{equation}
In some sense, we have solved the heat equation in the space of "point clouds" 
$(\mathbb{R}^{d})^N/\mathfrak{S}_N$, with initial position $A$, defined up to a permutation $\sigma\in\mathfrak{S}_N$ of the labels $i=1,\dots,N$.

\subsection{"Surfing" the "heat wave"}

After solving the heat equation in the space of "clouds" $(\mathbb{R}^{d})^N/\mathfrak{S}_N$ \eqref{eq:heat_equation}\eqref{eq:sym_initial_value}, we introduce the companion ODE in the space $(\R^d)^N$:
\begin{gather*}
\frac{\D X^\eps_t}{\D t}=v_\eps(t,X^\eps_t), \\
v_\eps(t,X)=-\frac{\eps}{2}\nabla\log\rho_{\eps}(t,X),
\end{gather*}
or, more explicitly
\begin{equation*}
v_\eps(t,X)=\frac{1}{2t}\frac{\displaystyle{\sum_{\sigma\in\mathcal{S}_N}(X-A^\sigma)\exp\left(-\frac{|X-A^\sigma|^2}{2\eps t}\right)}}{\displaystyle{\sum_{\sigma\in\mathcal{S}_N}\exp\left(-\frac{|X-A^\sigma|^2}{2\eps t}\right)}}=\frac{1}{2t}\left( X-\frac{\displaystyle{\sum_{\sigma\in\mathcal{S}_N}A^\sigma\exp\left(\frac{ X \cdot A^\sigma }{\eps t}\right)}}{\displaystyle{\sum_{\sigma\in\mathcal{S}_N}\exp\left(\frac{ X \cdot A^\sigma }{\eps t}\right)}}\right),
\end{equation*}
where if $U$ and $V$ are in $(\R^d)^N$, $U \cdot V$ denotes the inner product between $U$ and $V$. This velocity is chosen so that
\[
\frac{\partial \rho_{\eps}}{\partial t}(t,X) + \Div(\rho_{\eps}(t,X) v_{\eps}(t,X) ) = 0,
\]
\textit{i.e.} for the density $\rho_{\eps}$ to be transported by the velocity field $v_{\eps}$. We may solve this ODE for arbitrarily chosen position $X_{t_0}\in (\R^d)^N$ and initial time $t_0>0$.
In other words, we let the set of $N$ "particles" $X_t=(x_1(t), \dots x_N(t)) \in (\R^d)^N$  "surf" the "heat wave" generated by the lattice subject to Brownian agitation! By doing that, we just mimic the idea of quantum particles driven by the "onde pilote", as imagined by Louis de Broglie~\cite{deBroglie1927mecanique}
at the early stage of Quantum Mechanics.

\begin{Rem}
In that case, we would use the same ODE with $v=\eps\nabla\mathcal{I}m\log\psi$, 
$\psi$ solving the Schr\" odinger equation. For instance, we could consider
the free Schr\"odinger equation instead of the heat equation:
\begin{gather*}
\left(i\partial_t+\frac{\eps}{2}\Delta\right)\psi=0, \\
\psi(0,X)=\sum_{\sigma \in \mathfrak{S}_N} \exp\left(-\frac{|X-A^\sigma|^2}{a^2}\right),
\end{gather*}
with initial condition chosen according to "bosonic statistics".
However, in the quantum case, the analysis gets substantially more difficult, due to the possible
vanishing of the wave function $\psi$ during the evolution.
\end{Rem}

\subsection{Large deviations of the "heat wave" ODE}

Let us go back to the "heat wave" ODE and add a noise of the following type:
\begin{equation}
\label{eq:perturbated_companion}
\D X^{\eps,\eta}_t = v_{\eps}(t,X^{\eps, \eta}_t)\D t + \sqrt{\eta} \alpha(t) \D B_t,
\end{equation}
where $\eta$ is a positive number and $\alpha$ is a smooth function from $\R_+^*$ to $\R_+^*$. In other words, our "surfers" are now subject to some additional agitation, while surfing on the heat wave generated by the lattice already under Brownian agitation! 

We will see that when $\eta$ and $\eps$ are small, then the trajectories charged by the solution of this SDE that are in $P \in (\R^d)^N$ at time $t_0>0$ and in $Q \in (\R^d)^N$ at time $t_1>t_0$ (up to ordering) are very close to the dynamic of MAG.
Notice that the level of noise depends on time through the function $\alpha(t)$. It will be crucial in Subsection~\ref{subsec:change_of_time} since we will only recover MAG after suitable change of time.

Since, for fixed $\eps>0$ and $t>0$, $v_{\eps}$ is a smooth velocity field, existence of a strong solution and pathwise uniqueness for \eqref{eq:perturbated_companion} is standard once fixed a law for the initial position $X_{t_0}^{\eps, \eta}$, $t_0>0$. Furthermore, we may pass to the limit $\eta\rightarrow 0$, while $\eps>0$ is kept fixed, in the sense of large deviation: A direct application of classical Freidlin-Wentzell theory \cite{freidlin1998random,dembo2009large} leads to:
\begin{Thm}
\label{thm:LDP}
Let us fix $P,Q \in (\R^d)^N$ the endpoints of our trajectories, up to ordering, and $0<t_0<t_1$ two positive times. For fixed $\eps$ and as $\eta \downarrow 0$, the law of the solution of~\eqref{eq:perturbated_companion} between times $t_0$ and $t_1$ starting from $P$ and conditioned to arrive in $Q$ (up to ordering) satisfies the large deviation principle on $C^0([t_0,t_1]; (\R^d)^N)$ of good rate function $L_\eps$ defined for all $\XX = (X_t)_{t\in [t_0,t_1]}$ by:
\begin{equation*}
L_\eps(\XX) = \left\{
\begin{aligned}
&\int_{t_0}^{t_1} \frac{|\dot X_t - v_\eps(t, X_t)|^2}{\alpha(t)^2} \D t, && \mbox{if }\XX \in H^1([t_0,t_1]; (\R^d)^N),\\[-10pt]
&&& X_{t_0}\in\{P^\sigma\} \mbox{ and } X_{t_1} \in \{ Q^\sigma \},\\[5pt]
&+\infty, && \mbox{else},
\end{aligned}
\right.
\end{equation*}
where here and in the rest of the article, we denote by $\{ P^\sigma \}$ and $\{ Q^\sigma \}$ the sets $\{ P^\sigma, \, \sigma \in \perm \}$ and $\{ Q^\sigma, \, \sigma \in \perm \}$ respectively.
\end{Thm}

In the rest of the article, we will call $L_\epsilon$ the Freidlin-Wentzell action instead of the usual terminology "good rate function". Also, the endpoints $P$ and $Q$ are fixed once for all so we do not write explicitly the dependence of $L_\eps$ on those.

Theorem \ref{thm:LDP} asserts in particular that when $\eps$ is fixed and $\eta$ is small, if $X^{\eps, \eta}$ solves \eqref{eq:perturbated_companion}, given $X_{t_0}^{\eps, \eta} \in \{P^\sigma\}$ and $X^{\eps, \eta}_{t_1}\in\{Q^\sigma\}$, $X^{\eps,\eta}$ is with very high probability close to the minimizers of $L_\eps$. Now we will see that these functionals converge as $\eps \downarrow 0$ to a functional whose minimizers follow the dynamic of MAG\footnote{Once again up to a suitable change of time, see Subsection \ref{subsec:change_of_time}.}, in the sense of $\Gamma-$convergence (and hence in the sense of convergence of minimizers as well). As a consequence, when both $\eta$ and $\eps$ are small, given $X_{t_0}^{\eps, \eta} \in \{P^\sigma\}$ and $X^{\eps, \eta}_{t_1}\in \{Q^\sigma\}$, $X^{\eps,\eta}$ is close with high probability to the dynamic of MAG.

\begin{Rem}
Here, we chose to present the result for point clouds, \emph{i.e.} when the particle are still indistinguishable. However, the theorem could also be stated replacing the conditioning on $X_{t_0}\in\{P^\sigma\}$ and $X_{t_1} \in \{Q^\sigma \}$ by $X_{t_0}=P$ and $X_{t_1} = Q$. Otherwise stated, reintroducing distinguishable particles at this stage would not affect the results of this section (neither Theorem~\ref{thm:LDP} nor Theorem~\ref{mainresult} below). We decided to keep on working on clouds in order to avoid crossings of particles in Section~\ref{sec:sticky}. 
\end{Rem}

\subsection{The convergence result}
\label{subsec:main_result}

Define the following smooth convex function (see Lemma~\ref{lem:estim_g_eps} below):
\begin{equation}
\label{eq:def_f_eps}
\forall \eps > 0, \: \forall t > 0, \: \forall X \in (\R^d)^N,\qquad f_{\eps}(t,X) := \eps t \log \left[ \frac{1}{N!} \sum_{\sigma \in \perm} \exp \left(\frac{ X \cdot A^\sigma}{t \eps} \right)\right].\phantom{\forall \eps > 0, \: \forall t > 0, \: \forall X \in }
\end{equation}
It has the property that for all $\eps > 0$, $t > 0$, and $X \in (\R^d)^N$,
\[
v_{\eps}(t,X) = \frac{X - \nabla f_{\eps}(t,X)}{2t}.
\]
As a consequence, denoting by $\beta$ the smooth function $1/\alpha^2$, we can rewrite $L_{\eps}$ for all $\eps >0$ as:
\begin{equation*}
L_\eps(\XX) = \left\{
\begin{aligned}
&\int_{t_0}^{t_1} \left| \dot{X}_t - \frac{X_t - \nabla f_{\eps}(t, X_t)}{2t} \right|^2 \beta(t) \D t , && \mbox{if }\XX \in H^1([t_0,t_1]; (\R^d)^N),\\[-10pt]
&&& X_{t_0}\in\{P^\sigma\} \mbox{ and } X_{t_1} \in \{Q^\sigma\},\\[5pt]
&+\infty, && \mbox{else.}
\end{aligned}
\right.
\end{equation*}

When $\eps$ tends to zero, by virtue of the so-called Laplace's principle, we have the pointwise convergence:
\begin{equation}
\label{eq:deff}
\lim_{\eps \to 0 } f_{\eps}(t,X) = \max_{ \sigma \in \perm} X \cdot A^\sigma =: f(X).
\end{equation}
The function $f$ no longer depends on the time variable, and it is a convex function with finite values. As a consequence, for each $X \in (\R^d)^N$, the subdifferential $\partial f (X)$ of $f$ at $X$ is non-empty. We will consider the extended gradient $\overline\nabla f(X)$ of $f$ at $X$ defined as:
\begin{Def}[Extended gradient]
\label{def:extended_gradient}
We call \emph{extended gradient} of a real valued convex function $h$ at $X$, denoted by $\overline{\nabla} h (X)$, the element of $\partial h (X)$ with minimal Euclidean norm. 
\end{Def}
Here is our $\Gamma$-convergence result: 
\begin{Thm}
\label{mainresult}
As $\eps$ tends to $0$, the family of actions $(L_{\eps})_{\eps >0}$ $\Gamma-$converges to
\begin{equation*}
L(\XX) = \left\{
\begin{aligned}
&\int_{t_0}^{t_1} \left| \dot{X}_t - \frac{X_t - \overline{\nabla} f(X_t)}{2t} \right|^2 \beta(t) \D t , && \mbox{if }\XX \in H^1([t_0,t_1]; (\R^d)^N),\\[-10pt]
&&& X_{t_0}\in\{P^\sigma\} \mbox{ and } X_{t_1} \in \{Q^\sigma\},\\[5pt]
&+\infty, && \mbox{else.}
\end{aligned}
\right.
\end{equation*}
for the topology of uniform convergence of $C^0([t_0,t_1]; (\R^d)^N)$.
\end{Thm}
Theorem \ref{mainresult} can be seen as the main result of this article. In particular, it implies that any limit point as $\eps \downarrow 0$ of a sequence of minimizers of $L_{\eps}$ is a minimizer of $L$. So, one can rigorously obtain an effective action to describe the double limit 
$\lim_{\epsilon\downarrow 0}\lim_{\eta\downarrow 0}$ for the solution of the SDE \eqref{eq:perturbated_companion}. Notice that the lower semi-continuity of $L$ is a direct corollary of the $\Gamma-$convergence. In addition, the fact that $L$ has compact sublevels will be clear from the proof. Hence, the existence of global minimizers for $L$ (and hence for all the forthcoming functionals) follows from the direct method of calculus of variations.

We will prove Theorem \ref{mainresult} in Section~\ref{sec:proof} below, but before doing so, let us show that for a specific choice of $\beta$, we recover MAG.

\subsection{A regime where Monge-Amp\`ere gravitation arises}
\label{subsec:change_of_time}
Let us take $\beta(t) := t$ which corresponds to $\alpha(t) := 1/\sqrt{t}$. Through the change of variable:
\[
t = \exp(2 \theta), \quad Z_{\theta} = X_{\exp(2\theta)},
\]
we observe that for all $\XX \in C^0([t_0,t_1]; (\R^d)^N)$, $L(\XX) = \Lambda (\ZZ)$ with:
\begin{equation*}
\Lambda ( \ZZ ) = \left\{
\begin{aligned}
&\int_{\theta_0}^{\theta_1}\left| \dot{Z}_\theta - (Z_\theta - \overline{\nabla} f( Z_\theta)) \right|^2 \D \theta , && \mbox{if } \ZZ \in H^1([\theta_0, \theta_1]; (\R^d)^N),\\[-10pt]
&&& Z_{\theta_0} \in \{P^\sigma\} \mbox{ and } Z_{\theta_1} \in \{Q^\sigma\},\\[5pt]
&+\infty, && \mbox{else.}
\end{aligned}
\right.
\end{equation*}
(Recall the definition \eqref{eq:deff} of $f$.) Unexpectedly, this action is \textit{exactly} the one previously suggested by the third author in \cite{brenier2011modified} to include
dissipative phenomena (such as sticky collisions in one space dimension) in the Monge-Amp\`ere gravitational model!

It turns out to be equivalent to the following one:
\begin{equation*}
\Lambda' ( \ZZ ) = \left\{
\begin{aligned}
&\int_{\theta_0}^{\theta_1}\left\{ | \dot{Z}_\theta |^2 + |Z_\theta - \overline{\nabla} f( Z_\theta) |^2\right\} \D \theta , && \mbox{if } \ZZ \in H^1([\theta_0, \theta_1]; (\R^d)^N),\\[-10pt]
&&& Z_{\theta_0} \in \{P^\sigma\} \mbox{ and } Z_{\theta_1} \in \{Q^\sigma\},\\[5pt]
&+\infty, && \mbox{else.}
\end{aligned}
\right.
\end{equation*}
(By expanding the square and remarking that the mixed product is an exact temporal derivatives, so that its integral only involves the endpoints $P$ and $Q$.)

\subsection{Application of the least action principle}
\label{subsec:LAP}
We observe that the points $Z$ where $f$ is differentiable are those for which the maximum in the definition~\eqref{eq:deff} of $f$ is reached by a unique permutation $\sigma_{\mathrm{opt}}$ so that $\nabla f(Z)$ is nothing but $A^{\sigma_{\mathrm{opt}}}$. For such points $Z$, we get
\begin{equation*}
\frac{|Z-\nabla f(Z)|^2}{2}=\frac{|Z-A^{\sigma_{\mathrm{opt}}}|^2}{2}=
\frac{|Z|^2+|A^{\sigma_{\mathrm{opt}}}|^2}{2}-Z\cdot A^{\sigma_{\mathrm{opt}}}=\frac{|Z|^2+|A|^2}{2}-f(Z)
\end{equation*}
(by definition of $f$ and using that $|A^\sigma|=|A|$ for any $\sigma\in\mathfrak{S}_N$),
while, on the set $\mathcal N$ of non-differentiability of $f$, we rather have
\[
\frac{|Z-\overline\nabla f(Z)|^2}{2}< \frac{|Z|^2+|A|^2}{2}-f(Z).
\]
So the action we have obtained in the previous section, namely $\Lambda'$, bounds from below
\begin{equation*}
\Lambda^+ ( \ZZ ) = \left\{
\begin{aligned}
&\hspace{-2pt}\int_{\theta_0}^{\theta_1}\hspace{-5pt}\left\{ | \dot{Z}_\theta|^2 + \frac{|Z_\theta|^2+|A|^2}{2}- f(Z_{\theta}) \right\} \D \theta \D \theta , && \hspace{-5pt}\mbox{if } \ZZ \in H^1([\theta_0, \theta_1]; (\R^d)^N),\\[-10pt]
&&& Z_{\theta_0} \in \{P^\sigma\} \mbox{ and } Z_{\theta_1} \in \{Q^\sigma\},\\[5pt]
&+\infty, && \mbox{else.}
\end{aligned}
\right.
\end{equation*}
The second action is definitely strictly larger than the first one for those curves $\theta\rightarrow Z_\theta$  which take values in $\mathcal N$ (where $f$ is  not differentiable) on a set of times $\theta\in [\theta_0,\theta_1]$ which is not negligible for the Lebesgue measure. So, the least action principle may provide different optimal curves, depending on the action we choose. However, if a curve is optimal for $\Lambda'$ and almost surely takes value outside of $\mathcal N$, then it must also be optimal for $\Lambda^+$. Clearly, it is much easier to get the optimality equation for such a curve, by working with $\Lambda^+$ rather than with $\Lambda'$. By varying action $\Lambda^+$, we get, as optimality equation,
\begin{gather*}
\frac{\D\,^2 Z_\theta}{\D \theta^2} = Z_\theta - A^{\sigma_{\mathrm{opt}}}, \\
\sigma_{\mathrm{opt}} = \underset{\sigma \in \mathfrak{S}_N}{\mathrm{Arginf}} \| Z - A^\sigma \|^2,
\end{gather*}
which is the discrete dynamic announced in the introduction.

Of course, these equations have to be suitably modified for those curves which are optimal for action $\Lambda'$ but not for $\Lambda^+$ because they takes values in $\mathcal{N}$ for a non negligible amount of time. At this stage, we do not know how to do it. However, at least in the one-dimensional case $d=1$, such modifications are tractable and correspond to sticky collisions as $x_i(t) = x_j(t)$ occurs for different "particles" of labels $i \ne j$ and during interval of times of strictly positive Lebesgue measure, see Section \ref{sec:sticky}.

\section{Proof of the \texorpdfstring{$\boldsymbol{\Gamma-}$}-convergence}
\label{sec:proof}
The purpose of this this section is to prove Theorem~\ref{mainresult}. 
\subsection{The proof as a consequence of three lemmas}
As we will see, Theorem~\ref{mainresult} will be a consequence of three lemmas that we state below. Lemmas~\ref{lem:continuity_I} and~\ref{lem:Gamma_cv_K} both involve a family of smooth functions $(g_\eps)_{\eps >0}$ on $[\theta_0,\theta_1]\times\R^p$ for some $\theta_0 < \theta_1$ and $p \in \N$, pointwise converging to a function $g$. On these functions, we will assume the following:
\begin{Ass}
\label{ass:g_eps}
\begin{enumerate}[label=(H\arabic*)]
\item \label{H:zero_at_zero} For all $\eps>0$ and $\theta \in [\theta_0,\theta_1]$, $g_\eps(\theta, 0) = 0$.
\item \label{H:convex_eps} For all $\eps>0$ and $\theta \in [\theta_0,\theta_1]$, $g_\eps(\theta,\bullet)$ is convex.
\item \label{H:convex} For all $\theta \in [\theta_0,\theta_1]$, $g(\theta,\bullet)$ is convex, and the distributional derivative $\partial_\theta g$ is a $L^1_{\mathrm{loc}}$ function such that for all $\YY \in H^1([\theta_0, \theta_1]; (\R^d)^N)$, the map $\theta \mapsto g(\theta, Y_\theta)$ is also $H^1$, and for almost all $\theta \in [\theta_0, \theta_1]$, 
\begin{equation}
\label{eq:IPP_g}
\frac{\D \ }{\D \theta} g(\theta, Y_\theta) = \partial_\theta g(\theta, Y_\theta) + \overline\nabla g(\theta, Y_\theta) \cdot \dot Y_\theta.
\end{equation}  
\item \label{H:lip} The map $\nabla g_\eps$ is uniformly bounded, that is:
\begin{equation}
\label{eq:uniform_Lip}
L := \sup_{\eps>0} \sup_{\theta \in [\theta_0, \theta_1]} \sup_{Y \in \R^p} |\nabla g_\eps(\theta, Y)| < + \infty.
\end{equation}
\item \label{H:small_dt_nabla} The map $\partial_\theta \nabla g_\eps$ is uniformly bounded, that is:
\begin{equation}
\label{eq:small_dt_nabla}
M := \sup_{\eps>0} \sup_{\theta \in [\theta_0, \theta_1]} \sup_{Y \in \R^p} |\partial_\theta \nabla g_\eps(\theta, Y)| < + \infty.
\end{equation}
\end{enumerate}
\end{Ass}
In order to keep the proofs simple, we did not try to optimize these assumptions for Lemmas~\ref{lem:continuity_I} and~\ref{lem:Gamma_cv_K}, which are probably true in a far more general context. However, as we will see in the proof of Theorem~\ref{mainresult}, it suffices to check these assumptions for the family $(f_\eps)_{\eps>0}$ after suitable change of temporal and spatial scale. This is done in Lemma~\ref{lem:estim_g_eps}.

\begin{Lem}
\label{lem:continuity_I}
Let us consider $\theta_0< \theta_1 \in \R$, $\eta \in C^\infty([\theta_0,\theta_1]; \R_+^*)$ and a family $(g_\eps)_{\eps > 0}$ of smooth functions from $[\theta_0, \theta_1]\times \R^p$ to $\R$ pointwise converging to a function $g$, which satisfy~\ref{H:zero_at_zero}, \ref{H:convex}, \ref{H:lip} and~\ref{H:small_dt_nabla} from Assumptions~\ref{ass:g_eps}.
If a family of curves $(\YY^\eps)_{\eps>0}$ in $H^1([\theta_0, \theta_1]; \R^p)$ uniformly converges to a curve $\YY \in H^1([\theta_0,\theta_1]; \R^p)$, then
\begin{equation*}
\int_{\theta_0}^{\theta_1} \dot Y^\eps_\theta \cdot \nabla g_\eps(\theta, Y_\theta^\eps)\eta(\theta) \D \theta \underset{\eps \to 0}{\longrightarrow} \int_{\theta_0}^{\theta_1} \dot Y_\theta \cdot \overline\nabla g(\theta, Y_\theta)\eta(\theta) \D \theta.
\end{equation*}
\end{Lem}
\begin{Lem}
\label{lem:Gamma_cv_K}
Let us consider $\theta_0< \theta_1 \in \R$, $\eta \in C^\infty([\theta_0,\theta_1]; \R_+^*)$ and a family $(g_\eps)_{\eps > 0}$ of smooth functions from $[\theta_0, \theta_1]\times \R^p$ to $\R$ pointwise converging to a function $g$, and satisfying~\ref{H:convex_eps}, \ref{H:lip} and~\ref{H:small_dt_nabla} from Assumptions~\ref{ass:g_eps}.
Let us fix $R,S \in \R^p$ and define for $\eps>0$ and $\YY \in C^0([\theta_0, \theta_1]; \R^p)$:
\begin{gather*}
K_\eps(\YY) := \left\{\begin{aligned}  
&\frac{1}{2} \int_{\theta_0}^{\theta_1} \left\{ |\dot Y_\theta|^2 + |\nabla g_\eps(\theta , Y_\theta) |^2 \right\} \eta(\theta) \D \theta, &&\mbox{if }\YY \in H^1([\theta_1, \theta_1]; \R^p) \\[-10pt]
&&& Y_{\theta_0} =R \mbox{ and } Y_{\theta_1} = S,\\[5pt]
&+\infty, &&\mbox{else},
\end{aligned}
\right.\\
K(\YY) := \left\{\begin{aligned}  
&\frac{1}{2} \int_{\theta_0}^{\theta_1} \left\{ |\dot Y_\theta|^2 + |\overline\nabla g(\theta , Y_\theta) |^2 \right\} \eta(\theta) \D \theta, &&\mbox{if }\YY \in H^1([\theta_1, \theta_1]; \R^p) \\[-10pt]
&&& Y_{\theta_0} =R \mbox{ and } Y_{\theta_1} =S ,\\[5pt]
&+\infty, &&\mbox{else}.
\end{aligned}
\right.
\end{gather*}
Then $(K_\eps)_{\eps >0 }$ $\Gamma-$converges to $K$ for the topology on uniform convergence of $C^0([\theta_0, \theta_1]; \R^p)$.
\end{Lem}
\begin{Lem}
\label{lem:estim_g_eps}
With the notations of Theorem~\ref{mainresult}, let us call $\theta_0 := \log t_0 /2$, $\theta_1 := \log t_1 / 2$, $p=dN$ and for $\theta \in [\theta_0, \theta_1]$, $\eps>0$ and $Y \in (\R^d)^N$:
\begin{equation}
\label{eq:def_g_eps}
g_\eps(\theta, Y) := \frac{f_\eps(\exp(2\theta), \exp(\theta) Y)}{\exp(2\theta)} \qquad  \mbox{and} \qquad g(\theta, Y):= \frac{f(\exp(\theta) Y)}{\exp(2\theta)}.
\end{equation}  
Then $(g_\eps)_{\eps>0}$ pointwise converges to $g$, and they satisfy~\ref{H:zero_at_zero}, \ref{H:convex_eps}, \ref{H:convex}, \ref{H:lip} and \ref{H:small_dt_nabla} from Assumptions~\ref{ass:g_eps}.
\end{Lem}
In the next subsections, we will prove these three lemmas one by one. The most involved one is undoubtedly Lemma~\ref{lem:Gamma_cv_K}, which can be seen as the main step in the proof of Theorem~\ref{mainresult}. Let us start by proving Theorem~\ref{mainresult} using Lemmas~\ref{lem:continuity_I}, \ref{lem:Gamma_cv_K} and~\ref{lem:estim_g_eps}.
\begin{proof}[Proof of Theorem~\ref{mainresult}]
In this proof, the notation $\XX = X_t$ will stand for a generic curve from $[t_0, t_1]$ to $(\R^d)^N$. Associated with $\XX$, we define $\YY = Y_\theta$ the curve from $[\theta_0, \theta_1]$ to $(\R^d)^N$, where $\theta_0 := \log t_0/2$, $\theta_1 := \log t_1 / 2$, and for all $\theta \in [\theta_0, \theta_1]$, $Y_\theta := X_{\exp(2\theta)}/\exp(\theta)$. Note that $\XX$ is $H^1$ if and only if $\YY$ is $H^1$. If $(\XX^\eps)_{\eps>0}$ is a family of curves from $[t_0,t_1]$ to $(\R^d)^N$, we define in the same way the family of corresponding curves $(\YY^\eps)_{\eps>0}$ from $[\theta_0, \theta_1]$ to $(\R^d)^N$.

A quick computation shows that for all $\XX \in H^1([\theta_0, \theta_1]; (\R^d)^N)$, considering $\eta(\theta) := \beta(\exp(2(\theta))$ and $(g_\eps)_{\eps>0},g$ as defined in Lemma~\ref{lem:estim_g_eps}, we have:
\begin{align}
\label{eq:change_of_time_L_eps}
L_\eps(\XX) = \int_{t_0}^{t_1} \bigg| \dot{X}_t - &\frac{X_t - \nabla f_\eps(t,X_t)}{2t} \bigg|^2 \beta(t) \D t  =\frac{1}{2} \int_{\theta_0}^{\theta_1} \left| \dot Y_{\theta} + \nabla g_\eps(\theta, Y_\theta) \right|^2 \eta(\theta) \D \theta \\
\label{eq:decomposition_L_eps} &=  \frac{1}{2} \int_{\theta_0}^{\theta_1} \left\{ |\dot Y_\theta|^2 + |\nabla g_\eps(\theta , Y_\theta) |^2 \right\} \eta(\theta) \D \theta + \int_{\theta_0}^{\theta_1} \dot Y_\theta \cdot \nabla g_\eps(\theta, Y_\theta)\eta(\theta) \D \theta.
\end{align}
and:
\begin{align}
\notag L(\XX) = \int_{t_0}^{t_1} \bigg| \dot{X}_t - &\frac{X_t - \overline{\nabla} f(X_t)}{2t} \bigg|^2 \beta(t) \D t  =\frac{1}{2} \int_{\theta_0}^{\theta_1} \left| \dot Y_{\theta} + \overline\nabla g(\theta, Y_\theta) \right|^2 \eta(\theta) \D \theta \\
\label{eq:decomposition_L}&=  \frac{1}{2} \int_{\theta_0}^{\theta_1} \left\{ |\dot Y_\theta|^2 + |\overline\nabla g(\theta , Y_\theta) |^2 \right\} \eta(\theta) \D \theta + \int_{\theta_0}^{\theta_1} \dot Y_\theta \cdot \overline\nabla g(\theta, Y_\theta)\eta(\theta) \D \theta.
\end{align}
(Note that due to Lemma~\ref{lem:estim_g_eps}, $g$ is convex with respect to the space variable, and so $\overline\nabla g$ is well defined.)

\noindent\underline{Proof of the $\Gamma-\liminf$}. Let $\XX^\eps \underset{\eps \to 0}{\to} \XX$ for the topology of uniform convergence. Of course, we also have $\YY^\eps \underset{\eps \to 0}{\to} \YY$. Without loss of generality, we can suppose
\begin{equation}
\label{eq:limsup_finite}
\sup_{\eps > 0} L_\eps(\XX^\eps) < + \infty.
\end{equation} 
Indeed, if the $\liminf$ of this quantity is infinite, there is nothing to prove, and if the $\liminf$ is finite, up to an extraction, we can reduce ourselves to the case where the $\sup$ is finite.

As $\nabla g_\eps(\theta, Y)$ is bounded uniformly in $\eps, \theta, Y$ (this is~\ref{H:lip}), we easily deduce with~\eqref{eq:change_of_time_L_eps} that this assumption implies
\begin{equation*}
\sup_{\eps > 0} \int_{\theta_0}^{\theta_1} |\dot Y_\theta^\eps|^2 \D \theta < +\infty.
\end{equation*}
In particular, by lower semi-continuity of this $H^1$ seminorm with respect to  uniform convergence, all the curves $\YY^\eps$, $\eps>0$ as well as $\YY$ are in $H^1([\theta_0, \theta_1]; (\R^d)^N)$. In particular, applying Lemma~\ref{lem:continuity_I} thanks to Lemma~\ref{lem:estim_g_eps}, we have:
\begin{equation}
\label{eq:convergence_I}
\int_{\theta_0}^{\theta_1} \dot Y^\eps_\theta \cdot \nabla g_\eps(\theta, Y_\theta^\eps)\eta(\theta) \D \theta \underset{\eps \to 0}{\longrightarrow} \int_{\theta_0}^{\theta_1} \dot Y_\theta \cdot \overline\nabla g(\theta, Y_\theta)\eta(\theta) \D \theta.
\end{equation}

On the other hand, it is clear that under~\eqref{eq:limsup_finite}, for $\eps>0$ sufficiently small, the endpoints of $\XX^\eps$ are stationary, that is $X^\eps_{t_0} = P^{\sigma_0}$ and $X^\eps_{t_1} = Q^{\sigma_1}$ with $\sigma_0,\sigma_1$ independent of $\eps$. So for such $\eps$, $\YY^\eps$ satisfies the endpoint constraint for $K_\eps$ with $R := P^{\sigma_0}/\sqrt{t_0}$ and $S := Q^{\sigma_1}/\sqrt{t_1}$. Hence, applying Lemma~\ref{lem:Gamma_cv_K} thanks to Lemma~\ref{lem:estim_g_eps}, we have:
\begin{equation} 
\label{eq:lsc_K}
 \begin{aligned}
 \frac{1}{2} \int_{\theta_0}^{\theta_1} \bigg\{ |\dot Y_\theta|^2 + &|\overline\nabla g(\theta , Y_\theta) |^2 \bigg\} \eta(\theta) \D \theta = K(\YY) \\ &\leq \liminf_{\eps \to 0} K_\eps (\YY^\eps)= \liminf_{\eps \to 0}\frac{1}{2} \int_{\theta_0}^{\theta_1} \left\{ |\dot Y^\eps_\theta|^2 + |\nabla g_\eps(\theta , Y^\eps_\theta) |^2 \right\} \eta(\theta) \D \theta.
 \end{aligned}
 \end{equation}
 The result follows easily by gathering~\eqref{eq:decomposition_L_eps}, \eqref{eq:convergence_I}, \eqref{eq:lsc_K} and \eqref{eq:decomposition_L}.
 
 \noindent\underline{Proof of the $\Gamma-\limsup$}. Let $\XX \in C^0([t_0, t_1]; (\R^d)^N)$. Without loss of generality, we can suppose that $\XX \in H^1([t_0, t_1]; (\R^d)^N)$ and that it satisfies the endpoint constraint for $L$. In particular, $\YY$ belongs to $H^1([\theta_0, \theta_1]; (\R^d)^N)$ and satisfies the endpoint constraint for $K$ with $R := X_{t_0} / \sqrt{t_0}$ and $S := X_{t_1}/\sqrt{t_1}$. Lemmas~\ref{lem:Gamma_cv_K} and~\ref{lem:estim_g_eps} let us find a family $(\YY^\eps)_{\eps>0}$ converging to the corresponding $\YY$ such that:
 \begin{equation}
 \label{eq:gamma_lisump_K}
 \limsup_{\eps \to 0} K_\eps(\YY^\eps) \leq K(\YY).
 \end{equation}
 In particular $\YY^\eps$ is in $H^1$ for sufficiently small $\eps$, and by Lemmas~\ref{lem:continuity_I} and~\ref{lem:estim_g_eps}, 
 \begin{equation}
 \label{eq:continuity_I_limsup}
\int_{\theta_0}^{\theta_1} \dot Y^\eps_\theta \cdot \nabla g_\eps(\theta, Y_\theta^\eps)\eta(\theta) \D \theta \underset{\eps \to 0}{\longrightarrow} \int_{\theta_0}^{\theta_1} \dot Y_\theta \cdot \overline\nabla g(\theta, Y_\theta)\eta(\theta) \D \theta.
\end{equation}
The result follows easily from~\eqref{eq:decomposition_L_eps}, \eqref{eq:gamma_lisump_K}, \eqref{eq:continuity_I_limsup} and \eqref{eq:decomposition_L}, by noticing that because of \eqref{eq:gamma_lisump_K}, $\YY^\eps$ satisfies the endpoint constraint for $K_\eps$. Hence for such $\eps$, $\XX^\eps$ satisfies the endpoint constraint for $L_\eps$.
\end{proof}
\subsection{Proof of Lemma~\ref{lem:continuity_I}}
The proof of Lemma~\ref{lem:continuity_I} just consists in integrating by parts and using the convergence properties of $(g_\eps)_{\eps>0}$.
\begin{proof}[Proof of Lemma~\ref{lem:continuity_I}]
\underline{Integration by parts}. First, notice that as soon as $\YY \in H^1([\theta_0, \theta_1]; \R^p)$ and $\eps>0$, then $\theta \mapsto g_\eps(\theta, Y_\theta)$ and $\theta \mapsto g(\theta, Y_\theta)$ are also in $H^1$, with for almost every $\theta$:
\begin{equation*}
\frac{\D\ }{\D \theta}g_\eps(\theta, Y_\theta) = \partial_\theta g_\eps(\theta, Y_\theta) + \nabla g_\eps(\theta, Y_\theta)\cdot \dot Y_\theta \qquad \mbox{and} \qquad \frac{\D\ }{\D \theta}g(\theta, Y_\theta) = \partial_\theta g(\theta, Y_\theta) + \overline\nabla g(\theta, Y_\theta)\cdot \dot Y_\theta.
\end{equation*}
It is clear in the case of $g_\eps$ because $g_\eps$ is smooth, and it is the assumption~\ref{H:convex} in the case of~$g$. As a consequence, by an integration by parts, it suffices to prove that whenever $(\YY^\eps)_{\eps>0}$ converges to $\YY$ as $\eps \to 0$ for the topology of uniform convergence,
\begin{align*}
g_\eps(\theta_1, Y^\eps_{\theta_1})& \eta(\theta_1) - g_\eps(\theta_0, Y^\eps_{\theta_0}) \eta(\theta_0) - \int_{\theta_0}^{\theta_1} g_\eps(\theta,Y^\eps_\theta)\eta'(\theta) \D \theta - \int_{\theta_0}^{\theta_1} \partial_\theta g_\eps(\theta,Y^\eps_\theta)\eta(\theta) \D \theta\\
&\underset{\eps \to 0}{\longrightarrow} g(\theta_1, Y_{\theta_1}) \eta(\theta_1) - g(\theta_0, Y_{\theta_0}) \eta(\theta_0) - \int_{\theta_0}^{\theta_1} g(\theta,Y_\theta)\eta'(\theta) \D \theta - \int_{\theta_0}^{\theta_1} \partial_\theta g(\theta,Y_\theta)\eta(\theta) \D \theta.
\end{align*}

\noindent\underline{Convergence term by term}. The convergence
\begin{equation*}
g_\eps(\theta_1, Y^\eps_{\theta_1}) \eta(\theta_1) - g_\eps(\theta_0, Y^\eps_{\theta_0}) \eta(\theta_0) 
\underset{\eps \to 0}{\longrightarrow} g(\theta_1, Y_{\theta_1}) \eta(\theta_1) - g(\theta_0, Y_{\theta_0}) \eta(\theta_0)
\end{equation*}
is an easy consequence of the pointwise convergence and of the uniform Lipschitz bound~\ref{H:lip}.

For the same reason, we have for all $\theta \in [\theta_0, \theta_1]$, $g_\eps(\theta, Y^\eps_\theta)\underset{\eps \to 0}{\longrightarrow} g(\theta, Y_\theta)$. But on the other hand, because of~\ref{H:zero_at_zero} and~\ref{H:lip}, $g_\eps$ is locally bounded, uniformly in $\eps$. Hence, 
\begin{equation*}
 \int_{\theta_0}^{\theta_1} g_\eps(\theta,Y^\eps_\theta)\eta'(\theta) \D \theta
\underset{\eps \to 0}{\longrightarrow}  \int_{\theta_0}^{\theta_1} g(\theta,Y_\theta)\eta'(\theta) \D \theta 
\end{equation*} 
is a consequence of the dominated convergence theorem.

Because of~\ref{H:zero_at_zero} and~\ref{H:small_dt_nabla}, for all $\theta$, $(\partial_\theta g_\eps(\theta, \bullet))_{\eps>0}$ is compact for the topology of local uniform convergence. But its only possible limit point is the distributional derivative $\partial_\theta g$. As a consequence, $(\partial_\theta g_\eps)_{\eps>0}$ converges pointwise to $\partial_\theta g$, and because of the uniform bound~\ref{H:small_dt_nabla}, for all $\theta$, $\partial_\theta g_\eps(\theta, Y^\eps_\theta)\underset{\eps \to 0}{\longrightarrow} \partial_\theta g(\theta, Y_\theta)$. Because of~\ref{H:zero_at_zero} and~\ref{H:small_dt_nabla}, $\partial_\theta g_\eps$ is locally bounded, uniformly in $\eps$, and so 
\begin{equation*}
 \int_{\theta_0}^{\theta_1} \partial_\theta g_\eps(\theta,Y^\eps_\theta)\eta(\theta) \D \theta
\underset{\eps \to 0}{\longrightarrow}  \int_{\theta_0}^{\theta_1} \partial_\theta g(\theta,Y_\theta)\eta(\theta) \D \theta 
\end{equation*} 
is also a consequence of the dominated convergence theorem.
\end{proof}

\subsection{Proof of Lemma~\ref{lem:Gamma_cv_K}}
Before entering the proof of Lemma~\ref{lem:Gamma_cv_K}, we need to state a few standard results concerning the extended gradient $\overline\nabla$ as defined in Definition~\ref{def:extended_gradient}, and its links with the so-called \emph{resolvent map}. These tools could even be set in the infinite dimensional setting, that is in Hilbert spaces~\cite{stromberg1996operation}, or in metric spaces~\cite{ambrosio2008gradient}, and we refer to these works for the proofs.

Consider $h : \R^p \to \R$ a convex function. It is easily shown that for all $X \in \R^p$, 
\begin{equation*}
|\overline{\nabla} h(X)| = \sup_{Y \neq X} \frac{(h(X) - h(Y))_+}{|X-Y|}.
\end{equation*}
The following proposition, that we state without a proof, is an easy consequence of this formula and of the elementary fact that in finite dimension, pointwise convergence of convex functions to a finite valued convex functions implies $\Gamma-$convergence.
\begin{Prop}
	\label{prop:lsc_slope}
	Let $(h_\eps)_{\eps >0}$ be a family of convex functions on $\R^p$ pointwise converging to $h$, and let $(X^\eps)_{\eps>0}$ be a family of points in $\R^p$ converging to $X$. Then
	\begin{equation*}
	|\overline\nabla h(X)| \leq \liminf_{\eps \to 0} |\overline\nabla h_\eps(X^\eps)|.
	\end{equation*}
\end{Prop}

For $\tau>0$ and $X \in \R^p$, define the resolvent operator by:
\begin{equation*}
J_{\tau, h}(X) := \underset{Y \in \R^p}{\mbox{argmin }} h(Y) + \frac{|Y-X|^2}{2\tau}.
\end{equation*}
Once again, the following proposition is standard, and we state it here without a proof.
\begin{Prop}
	\label{prop:resovent}
	\begin{enumerate}
	\item We have for all $X \in \R^p$ and $\tau>0$:
	\begin{equation}
	\label{eq:estim_slope}
	|\overline\nabla h(J_{\tau,h}(X)) | \leq \left| \frac{X - J_{\tau, h}(X)}{\tau} \right| \leq |\overline\nabla h(X)|.
	\end{equation}
	
	\item If $h$ is everywhere differentiable and $X \in \R^p$, then the following first order condition holds:
	\begin{equation*}
	 \frac{X - J_{\tau, h}(X)}{\tau} = \nabla h(J_{\tau,h}(X)).
	\end{equation*}
	
	\item If $(h_\eps)_{\eps >0}$ is a family of convex functions on $\R^p$ pointwise converging to $h$, then for all $\tau>0$ and $X \in \R^p$,
	\begin{equation}
	\label{eq:convergence_resolvent}
	J_{\tau, h_\eps}(X) \underset{\eps \to 0}{\longrightarrow} J_{\tau,h}(X).
	\end{equation}
	\end{enumerate}
\end{Prop}

We are now ready for the proof of Lemma~\ref{lem:Gamma_cv_K}.
\begin{proof}[Proof of Lemma~\ref{lem:Gamma_cv_K}]
	\noindent\underline{Proof of the $\Gamma-\liminf$}.	It is straightforward using Fatou's lemma, Proposition~\ref{prop:lsc_slope} and the lower semi-continuity of $\YY \mapsto \int_{\theta_0}^{\theta_1} |\dot Y_\theta|^2\D \theta$ with respect to the topology of uniform convergence.
	
	\noindent\underline{Proof of the $\Gamma-\limsup$}. Let us consider a curve $\YY \in H^1([\theta_0, \theta_1]; \R^p)$ with $Y_{\theta_0} = R$ and $Y_{\theta_1}=S$ (else there is nothing to prove). For all $\eps>0$ and $\tau>0$, we define:
	\begin{equation*}
	\YY^{\tau, \eps}: \theta \mapsto J_{\tau, g_\eps(\theta,\bullet)}(Y_\theta),
	\end{equation*}
	and correspondingly:
	\begin{equation*}
	\YY^{\tau}: \theta \mapsto J_{\tau, g(\theta,\bullet)}(Y_\theta).
	\end{equation*}
	First, we prove:
	\begin{equation}
	\label{eq:first_estim_gamma_limsup}
	\limsup_{\tau \to 0} \limsup_{\eps \to 0} \frac{1}{2} \int_{\theta_0}^{\theta_1} \left\{ |\dot Y^{\tau,\eps}_\theta|^2 + |\nabla g_\eps(\theta , Y^{\tau, \eps}_\theta) |^2 \right\} \eta(\theta) \D \theta \leq \frac{1}{2} \int_{\theta_0}^{\theta_1} \left\{ |\dot Y_\theta|^2 + |\overline\nabla g(\theta , Y_\theta) |^2 \right\} \eta(\theta) \D \theta.
	\end{equation}
	We will then choose $\tau$ as a function of $\eps$ and show how to fix the endpoints.
	
	\noindent\underline{Proof of \eqref{eq:first_estim_gamma_limsup}}. By the second point of Proposition~\ref{prop:resovent}, for all $\eps,\tau, \theta$, we have:
	\begin{equation*}
	Y_\theta = Y^{\tau, \eps}_\theta + \tau \nabla g_\eps(\theta, Y^{\tau, \eps}_\theta).
	\end{equation*}
	Using the smoothness and convexity of $g_\eps$, and $\YY \in H^1$, we easily deduce that $\YY^{\tau, \eps}$ is in $H^1$ and that for almost all $\theta$,
	\begin{equation*}
	\dot Y_\theta = \Big(\mathbb{I} + \tau \mathrm{D}^2g_\eps(\theta, Y^{\tau, \eps}_\theta)\Big) \cdot \dot Y^{\tau, \eps}_\theta + \tau \partial_\theta \nabla g_\eps(\theta, Y^{\tau, \eps}_\theta).
	\end{equation*}
	By convexity of $g_\eps$, we have $\mathbb{I} \leq \mathbb{I} + \tau \mathrm{D}^2 g_\eps$ in the sense of symmetric matrices, and hence:
	\begin{equation}
	\label{eq:estim_velocity}
	| \dot Y^{\tau, \eps}_\theta| \leq |\dot Y_\theta - \tau \partial_\theta \nabla g_\eps (\theta, Y^{\tau,\eps}_\theta)| \leq |\dot Y_\theta| + \tau M.
	\end{equation}
	Recall that $M$ was defined in the uniform integrability assumption~\eqref{eq:small_dt_nabla} on $\partial_\theta \nabla g_\eps$. (In the case when $\partial_\theta \nabla g_\eps = 0$, we recover the known fact that for $h$ independent of time, $J_{\tau,h}$ is contractive.) Then, we deduce:
	\begin{align*}
	\limsup_{\eps \to 0}\frac{1}{2} \int_{\theta_0}^{\theta_1} &\left\{ |\dot Y^{\tau,\eps}_\theta|^2 + |\nabla g_\eps(\theta , Y^{\tau, \eps}_\theta) |^2 \right\} \eta(\theta) \D \theta 
	\\&\overset{\eqref{eq:estim_slope}\eqref{eq:estim_velocity}}{\leq} \limsup_{\eps \to 0}\frac{1}{2} \int_{\theta_0}^{\theta_1} \left\{ \bigg( |\dot Y_\theta| + \tau M \bigg)^2 + \left|\frac{Y_\theta - Y^{\tau, \eps}_\theta}{\tau} \right|^2 \right\} \eta(\theta) \D \theta\\
	&\hspace{5pt}\overset{\eqref{eq:convergence_resolvent}}{\leq} \frac{1}{2}\int_{\theta_0}^{\theta_1}\left\{\Big( |\dot Y_\theta| + \tau M \Big)^2 + \left|\frac{Y_\theta - Y^{\tau}_\theta}{\tau} \right|^2 \right\}\eta(\theta)\D \theta \\
	&\hspace{5pt}\overset{\eqref{eq:estim_slope}}{\leq}\frac{1}{2}\int_{\theta_0}^{\theta_1}\left\{ \Big( |\dot Y_\theta| + \tau M \Big)^2 + \left|\overline\nabla g(\theta, Y_\theta) \right|^2 \right\}\eta(\theta)\D \theta.
	\end{align*}
 Formula~\eqref{eq:first_estim_gamma_limsup} follows.
	
	\noindent\underline{Choice of $\tau = \tau(\eps)$}. Because of~\eqref{eq:first_estim_gamma_limsup}, and because:
	\begin{equation*}
	\forall \eps>0, \qquad Y^{\tau, \eps}_{\theta_0} \underset{\tau \to 0}{\longrightarrow} R \qquad \mbox{and}\qquad Y^{\tau, \eps}_{\theta_1} \underset{\tau \to 0}{\longrightarrow} S,
	\end{equation*}
	it is possible to find a non-increasing function $\tau = \tau(\eps)$ converging sufficiently slowly to $0$ so that:
	\begin{gather}
	\label{eq:second_estil_Gamma_limsup}\limsup_{\eps \to 0} \frac{1}{2} \int_{\theta_0}^{\theta_1} \left\{ |\dot Y^{\tau(\eps),\eps}_\theta|^2 + |\nabla g_\eps(\theta , Y^{\tau(\eps), \eps}_\theta) |^2 \right\} \eta(\theta) \D \theta \leq \frac{1}{2} \int_{\theta_0}^{\theta_1} \left\{ |\dot Y_\theta|^2 + |\overline\nabla g(\theta , Y_\theta) |^2 \right\} \eta(\theta) \D \theta,\\
	\label{eq:convergence_endpoints} Y^{\tau(\eps), \eps}_{\theta_0} \underset{\eps \to 0}{\longrightarrow} R \qquad \mbox{and}\qquad Y^{\tau(\eps), \eps}_{\theta_1} \underset{\tau \to 0}{\longrightarrow} S.
	\end{gather}
	
	\noindent \underline{Fixing the endpoints}. For fixed $\eps$ and small $\delta>0$, we will define $\ZZ^{\delta,\eps}$ as a slight modification of the curve $\YY^{\eps, \tau(\eps)}$ in such a way that $\ZZ^{\delta,\eps}$ joins $R$ to $S$. For this, we just set for $\theta \in [\theta_0, \theta_1]$:
	\begin{equation*}
	Z^{\delta,\eps}_\theta = \left\{\begin{aligned}
	& R + \frac{\theta - \theta_0}{\delta} \Big( Y^{\tau(\eps), \eps}_{\delta_0 + \delta} - R \Big), && \mbox{if }\theta \in [\theta_0, \theta_0 + \delta],\\
	&Y^{\tau(\eps), \eps}_\theta, &&\mbox{if } \theta \in [\theta_0 + \delta, \theta_1 - \delta],\\
	&S + \frac{\theta_1 - \theta}{\delta} \Big( Y^{\tau(\eps), \eps}_{\delta_1 - \delta} - S \Big), && \mbox{if }\theta \in [\theta_1 - \delta, \theta_1 ],
	\end{aligned}
	\right.
	\end{equation*}
	A quick computation shows:
	\begin{equation}
	\label{eq:third_estim_Gamma_limsup}
	\begin{aligned}
	&\frac{1}{2} \int_{\theta_0}^{\theta_1} \left\{ |\dot Z^{\delta,\eps}_\theta|^2 + |\nabla g_\eps(\theta , Z^{\delta, \eps}_\theta) |^2 \right\} \eta(\theta) \D \theta \\
	\leq \frac{1}{2}& \int_{\theta_0}^{\theta_1} \left\{ |\dot Y^{\tau(\eps),\eps}_\theta|^2 + |\nabla g_\eps(\theta , Y^{\tau(\eps), \eps}_\theta) |^2 \right\} \eta(\theta) \D \theta +  \| \eta \|_{\infty} \left(\frac{|Y^{\tau(\eps), \eps}_{\theta_0 + \delta} - R|^2}{2\delta} +\frac{|Y^{\tau(\eps), \eps}_{\theta_1 - \delta} - S|^2}{2\delta} + \delta L^2\right),
	\end{aligned}
	\end{equation}
	where $L$ is defined in the uniform Lipschitz assumption~\eqref{eq:uniform_Lip} for $g_\eps$. 
	
Let us estimate $|Y^{\tau(\eps), \eps}_{\theta_0 + \delta} - R|^2/2\delta$. We have:
\begin{equation*}
\frac{|Y^{\tau(\eps), \eps}_{\theta_0 + \delta} - R|^2}{2\delta} \leq \frac{|Y^{\tau(\eps), \eps}_{\theta_0 } - R|^2}{\delta} + \frac{|Y^{\tau(\eps), \eps}_{\theta_0 + \delta} - Y^{\tau(\eps), \eps}_{\theta_0 }|^2}{\delta} \leq \frac{|Y^{\tau(\eps), \eps}_{\theta_0 } - R|^2}{\delta} + \int_{\theta_0}^{\theta_0 + \delta} |\dot Y^{\tau(\eps),\eps}_\theta|^2 \D \theta.
\end{equation*}
Because of \eqref{eq:estim_velocity}, \eqref{eq:small_dt_nabla} and $\YY \in H^1$, the integral $\int_{\theta_0}^{\theta_0 + \delta} |\dot Y^{\tau(\eps),\eps}_\theta|^2 \D \theta \to 0$ as $\delta \to 0$, uniformly in $\eps$: we bound it by a function $v_i = v_i(\delta)$ tending to $0$ as $\delta \to 0$. In the same way,
\begin{equation*}
\frac{|Y^{\tau(\eps), \eps}_{\theta_1 - \delta} - S|^2}{2\delta} \leq \frac{|Y^{\tau(\eps), \eps}_{\theta_1 } - S|^2}{\delta} + v_f(\delta),
\end{equation*}
where $v_f(\delta) \to 0$ as $\delta \to 0$.

Plugging these bounds into~\eqref{eq:third_estim_Gamma_limsup}, we get:
\begin{equation*}
\begin{aligned}
\frac{1}{2} \int_{\theta_0}^{\theta_1} &\left\{ |\dot Z^{\delta,\eps}_\theta|^2 + |\nabla g_\eps(\theta , Z^{\delta, \eps}_\theta) |^2 \right\} \eta(\theta) \D \theta \\
	&\leq \frac{1}{2} \int_{\theta_0}^{\theta_1} \left\{ |\dot Y^{\tau(\eps),\eps}_\theta|^2 + |\nabla g_\eps(\theta , Y^{\tau(\eps), \eps}_\theta) |^2 \right\} \eta(\theta) \D \theta +  \| \eta \|_{\infty} \left(\frac{u(\eps)}{\delta}+ v(\delta)\right),
	\end{aligned}
\end{equation*}
where $u(\eps) := |Y^{\tau(\eps), \eps}_{\theta_0 } - R|^2 + |Y^{\tau(\eps), \eps}_{\theta_1 - \delta} - S|^2 \to 0$ as $\eps \to 0$ by~\eqref{eq:convergence_endpoints}, and $v(\delta) := v_i(\delta) + v_f(\delta) + \delta L^2 \to 0$ as $\delta \to 0$. Hence, choosing $\delta(\eps) := \sqrt{u(\eps)}$, we find with the help of~\eqref{eq:second_estil_Gamma_limsup} that $\ZZ^{\delta(\eps), \eps}$ is a recovery sequence for the $\Gamma-\limsup$ of $K_\eps$ towards $K$. 
\end{proof}

\subsection{Proof of Lemma~\ref{lem:estim_g_eps}}
The proof is straightforward, and relies on explicit computations.
\begin{proof}[Proof of Lemma~\ref{lem:estim_g_eps}]
Let us define for $X \in (\R^d)^N$:
\begin{equation}
\label{eq:def_h_gamma_cv}
h(X) := \log \left[ \frac{1}{N!} \sum_{\sigma \in \perm} \exp ( X \cdot A^\sigma)\right].
\end{equation}
For $\eps >0$, $\theta \in [\theta_0, \theta_1]$ and $Y \in (\R^d)^N$, we have by definition of $f_\eps$ and $g_\eps$ (formulas~\eqref{eq:def_f_eps} and~\eqref{eq:def_g_eps} respectively):
\begin{equation}
\label{eq:link_g_eps_h}
g_\eps(\theta, Y) = \eps h\left( \frac{Y}{\eps \exp(\theta)} \right).
\end{equation}

\noindent\underline{Proof of~\ref{H:zero_at_zero}}. It is obvious.

\noindent\underline{Proof of~\ref{H:convex_eps}}. By~\eqref{eq:link_g_eps_h}, it suffices to check that $h$ is convex. Differentiating twice~\eqref{eq:def_h_gamma_cv}, we get for all $X \in (\R^d)^N$:
\begin{equation}
\label{eq:D2h}
\mathrm{D}^2 h(X) = \cg A^\sigma \otimes A^\sigma \cd_X - \cg A^\sigma \cd_X \otimes \cg A^\sigma \cd_X = \cg A^\sigma - \cg A^\sigma \cd_X \cd_X \otimes \cg A^\sigma - \cg A^\sigma \cd_X \cd_X,
\end{equation}
where if $a$ is a function of $\sigma$, $\cg a(\sigma) \cd_X$ stands for:
\begin{equation*}
\cg a(\sigma) \cd_X := \frac{\displaystyle{\sum_{\sigma \in \perm}} a(\sigma) \exp(X\cdot A^\sigma)}{\displaystyle{\sum_{\sigma \in \perm}} \exp(X\cdot A^\sigma)}.
\end{equation*}
It follows that $\mathrm{D}^2 h(X)$ is a nonnegative symmetric matrix.

\noindent \underline{Proof of \ref{H:convex}}. By the definitions~\eqref{eq:deff} of $f$ and~\eqref{eq:def_g_eps} of $g$, we have for all $\theta \in [\theta_0, \theta_1]$ and $Y \in (\R^d)^N$:
\begin{equation*}
g(\theta, Y) = \frac{f(Y)}{\exp(\theta)}.
\end{equation*}
The convexity is obvious, let us check~\eqref{eq:IPP_g}. Let us consider $\YY \in H^1([\theta_0, \theta_1]; (\R^d)^N)$. The function $g$ is clearly locally Lipschitz in both $\theta$ and $Y$. As a consequence, the map $G: \theta \mapsto g(\theta, Y_\theta)$ is also $H^1$. Let us take $\theta \in (\theta_0, \theta_1)$ a point where both $\YY$ and $G$ are differentiable (this happens for almost every $\theta$). We have:
\begin{align*}
G'(\theta) &= \lim_{\delta \downarrow 0} \frac{1}{\delta} \left\{ \frac{f(Y_{\theta + \delta})}{\exp(\theta + \delta)} - \frac{f(Y_\theta)}{\exp(\theta)} \right\} = -\frac{f(Y_\theta)}{\exp(\theta)} + \frac{1}{\exp(\theta)}\lim_{\delta \downarrow 0} \frac{f(Y_{\theta + \delta}) - f(Y_\theta)}{\delta} \\
&\geq -\frac{f(Y_\theta)}{\exp(\theta)} + \frac{1}{\exp(\theta)}\limsup_{\delta \downarrow 0} \overline\nabla f(Y_\theta) \cdot \frac{Y_{\theta + \delta} - Y_\theta}{\delta}\\
&= \partial_\theta g(\theta, Y_\theta) + \frac{\overline\nabla f(Y_\theta)}{\exp(\theta)} \cdot \dot Y_\theta = \partial_\theta g(\theta, Y_\theta) + \overline\nabla g(\theta, Y_\theta) \cdot \dot Y_\theta,
\end{align*}
where we used $f(Y_{\theta + \delta}) \geq f(Y_\theta) + \overline\nabla f(Y_\theta)\cdot (Y_{\theta + \delta} - Y_\theta)$ to get the second line. In the same way, we have:
\begin{equation*}
G'(\theta) = \lim_{\delta \downarrow 0} \frac{1}{\delta} \left\{ \frac{f(Y_{\theta})}{\exp(\theta )} - \frac{f(Y_{\theta-\delta})}{\exp(\theta - \delta)} \right\} \leq \partial_\theta g(\theta, Y_\theta) + \overline\nabla g(\theta, Y_\theta) \cdot \dot Y_\theta.
\end{equation*}
The result follows from gathering these two inequalities. 

\noindent \underline{Proof of \ref{H:lip}}. In view of~\eqref{eq:link_g_eps_h} and as $\theta_0> - \infty$, it suffices to check that $\nabla h$ is bounded. Differentiating~\eqref{eq:def_h_gamma_cv} at $X \in (\R^d)^N$ leads to:
\begin{equation*}
\nabla h(X) = \cg A^\sigma \cd_X,
\end{equation*}
which is clearly bounded by $|A|$.

\noindent \underline{Proof of \ref{H:small_dt_nabla}}. Using~\eqref{eq:link_g_eps_h}, we get for all $\eps>0$, $\theta \in [\theta_0, \theta_1]$ and $Y \in (\R^d)^N$:
\begin{equation*}
\partial_\theta \nabla g_\eps(\theta, Y) = -\frac{1}{\exp(\theta)}\left( \nabla h\left( \frac{Y}{\eps \exp(\theta)}\right) + \mathrm{D}^2h \left( \frac{Y}{\eps \exp(\theta)}\right) \cdot \frac{Y}{\eps \exp(\theta)} \right).
\end{equation*}
As we already saw in~\ref{H:lip} that $\nabla h$ is bounded, it suffices to prove that $X \mapsto \mathrm{D}^2h(X) \cdot X$ is bounded. Let us expand everything in~\eqref{eq:D2h} and apply $X$ to the right. We get:
\begin{equation*}
\mathrm{D}^2 h(X) \cdot X = \frac{\displaystyle{\sum_{\sigma,\eta \in \perm}}X \cdot (A^\sigma - A^\eta) A^\sigma \exp\Big(X \cdot (A^\sigma + A^\eta)\Big)}{\displaystyle{\sum_{\sigma,\eta \in \perm}} \exp\Big(X \cdot (A^\sigma + A^\eta)\Big)}.
\end{equation*}
As a consequence, it suffices to show that for each $\sigma,\eta \in \perm$,
\begin{equation*}
T(\sigma, \eta, X) := \frac{X \cdot (A^\sigma - A^\eta)\exp\Big(X \cdot (A^\sigma + A^\eta)\Big)}{\displaystyle{\sum_{\sigma',\eta' \in \perm}} \exp\Big(X \cdot (A^{\sigma'} + A^{\eta'})\Big)}
\end{equation*} 
is bounded, uniformly in $X$. First, if $\eta = \sigma$, then $T(\sigma, \sigma, X)=0$. Else, let us use the bound:
\begin{equation*}
\sum_{\sigma',\eta' \in \perm} \exp\Big(X \cdot (A^{\sigma'} + A^{\eta'})\Big) \leq \exp\Big( 2 X \cdot A^\sigma \Big)+ \exp\Big( 2 X\cdot A^\eta\Big),
\end{equation*}
obtained by only keeping the terms corresponding to $\sigma'= \eta' = \sigma$ and $\sigma' = \eta' = \eta$ in the sum. This leads to:
\begin{equation*}
|T(\sigma, \eta, X)| \leq \frac{|X \cdot (A^\sigma - A^\eta)|\exp\Big(X \cdot (A^\sigma + A^\eta)\Big)}{\exp\Big( 2 X \cdot A^\sigma \Big)+ \exp\Big( 2 X\cdot A^\eta\Big)} =  \frac{|X \cdot (A^\sigma - A^\eta)|}{\exp\Big( - |X \cdot (A^\sigma - A^\eta)| \Big)+ \exp\Big( |X \cdot (A^\sigma - A^\eta)|\Big)},
\end{equation*}
which is clearly bounded uniformly in $X$. The result follows.
\end{proof}
\section{The case of dimension 1: sticky collisions}
\label{sec:sticky}
In this section, we will study the global minimizers of the functional $\Lambda'$ obtained in Subsection~\ref{subsec:change_of_time}, in dimension $d=1$. If we call $t$ the time variable and if we replace $\theta_0$ and $\theta_1$ by $0$ and $T$ respectively, due to the invariance of the functional through translation in time, $\Lambda'$ reads: 
\begin{equation}
\label{eq:functional_dim1}
\Lambda' ( \ZZ ) = \left\{
\begin{aligned}
&\int_{0}^{T}\left\{ | \dot{Z}_t |^2 + |Z_t - \overline{\nabla} f( Z_t) |^2\right\} \D t , && \mbox{if } \ZZ \in H^1([0, T]; \R^{N}),\\[-10pt]
&&& Z_{0} \in\{ P^\sigma \} \mbox{ and } Z_{T} \in \{ Q^\sigma \} ,\\[5pt]
&+\infty, && \mbox{else,}
\end{aligned}
\right.
\end{equation}
where:
\begin{equation}
\label{eq:def_f_1d}
f(X) = \max_{ \sigma \in \perm} X \cdot A^{\sigma}, \qquad X \in \R^N.
\end{equation}
Here, we chose a \emph{strictly ordered} $A=(a_1,\dots,a_N)$, that is such that $a_1 < \dots <a_N$, $P,Q \in \R^N$ and $T>0$. Once again, when $X=(x_1, \dots, x_N) \in \R^N$ and $\sigma \in \perm$, $X^\sigma := (x_{\sigma(1)}, \dots, x_{\sigma(N)})$, and $\{P^\sigma\}$ and $\{  Q^\sigma\}$ refer to $\{P^\sigma,\, \sigma \in \perm\}$ and $\{Q^\sigma,\, \sigma \in \perm \}$ respectively. Of course $P=(p_1, \dots, p_N)$ and $Q=(q_1,\dots,q_N)$ can be supposed to be ordered, that is $p_1 \leq \dots \leq p_N$ and $q_1 \leq \dots \leq q_N$. We recall that we defined the \emph{extended gradient} $\overline{\nabla} f$ in Definition~\ref{def:extended_gradient}. As already noticed in Subsection~\ref{subsec:main_result}, the existence of global minimizers for $\Lambda'$ follows from the direct method of calculus of variations. Uniqueness does not hold in general, even up to permutations.

The purpose of the section is twofold. On the one hand, we will show that the model has nice regularity properties: any global minimizer of $\Lambda'$ is smooth except on a finite number of "sticking" or separation" times\footnote{Notice that $\Lambda'$ is invariant under time inversion, so that if particles are allowed to stick, they are also allowed to separate.}. On the other hand, we will justify as claimed in Section~\ref{sec:derivation} that $\Lambda'$ describes a model with sticky collisions in the sense that a minimizer $\ZZ= (z_1(t),\dots, z_N(t))$ of $\Lambda'$ will typically exhibit some sticking effects as $z_i(t) = z_j(t)$ for $i \neq j$ on non-trivial intervals. 

To describe the sticking effect, it is convenient to introduce the following definition:

\begin{Def}[Partition of $\llbracket 1,N \rrbracket$]
\label{def:partition}
Let $X \in \R^N$. We say that $X$ is divided according to $\pi(X)$ when $\pi(X)$ is the partition of $\llbracket 1, N \rrbracket$ induced by the relation:
\[
\forall (i,j) \in \llbracket 1, N \rrbracket^2, \quad i \sim j \quad \Leftrightarrow \quad x_i = x_j.
\]
We call $C(X,i)$ the class of $i \in \llbracket 1, N \rrbracket$ in $\pi(X)$.
\end{Def}

The main result of the section is the following result:
\begin{Thm}[Regularity of the optimal trajectories]
\label{thm:traj_reg}
For given $A,P,Q\in \R^N$ and $T>0$ as before, let $\ZZ$ be a global minimizer of $\Lambda'$ defined in \eqref{eq:functional_dim1}. Then $\ZZ$ is continuous and there exist:
\begin{equation*}
0=t_0 < t_1<\dots < t_p = T
\end{equation*}
a family of times such that for each $i = 1,\dots,p$, $\ZZ$ is smooth on $[t_{i-1},t_i]$, and $\pi(\ZZ)$ is constant on $(t_{i-1},t_i)$.
\end{Thm}
It will be quite clear from the proof that sticking effects do occur. This exactly means that there exist trajectories $\ZZ$ for which with the notations of Subsection \ref{subsec:LAP}, $\Lambda'(\ZZ) < \Lambda^+(\ZZ)$. For such trajectories, $Z_t$ is located on the set where $f$ is not differentiable for a set of times of positive Lebesgue measure. But in dimension $1$, this set is exactly the set where at least two particles are located at the same place. Otherwise stated, the set of times when $\pi(\ZZ) \neq \{\{1\},\dots,\{N\}\}$ is typically of positive Lebesgue measure. As a consequence of Theorem \ref{thm:traj_reg}, it is even a finite union of intervals.

Still it might be convenient to illustrate the sticking effects included in the model by the following easy proposition. It asserts that the set of times when all the particles are stuck is an interval: if all the particles are stuck at two different times, the cheapest behaviour between these two times is to remain stuck. It also shows that this phenomenon occurs: if all the particles are sufficiently close at the initial and final time, then they necessarily stick together during a non-trivial interval along the evolution. 

\begin{Prop}[Intervals of full degeneration]
\label{prop:full_degenerescence}
\begin{enumerate}
\item For given $A,P,Q\in \R^N$ and $T>0$ as before, let $\ZZ = (z_1(t),\dots,z_N(t))$ be a global minimizer of $\Lambda'$. Suppose there exist two times $0\leq t_1 < t_2 \leq T$ such that:
\begin{equation*}
z_1(t_1) = \dots =z_N(t_1) \quad \mbox{and} \quad  z_1(t_1) = \dots =z_N(t_2).
\end{equation*}
Then for all $t \in [t_1,t_2]$, $z_1(t) = \dots = z_N(t)$.
\item For given $A \in \R^N$ and $T>0$ as before, the set $\mathcal{U}$ of endpoints $P,Q \in \R^N$ with the property that for all minimizer $\ZZ = (z_1(t),\dots ,z_N(t))$ of $\Lambda'$, the set of times:
\begin{equation*}
\big\{ t \in [0,T] \, \big| \, z_1(t) = \dots = z_N(t) \big\}
\end{equation*}
is a non-trivial interval, is a neighbourhood of $\{ P,Q \in \R^N \, | \, p_1 = \dots = p_N \mbox{ and } q_1 = \dots =q_N  \}$.
\end{enumerate}
\end{Prop}

The proof of Proposition \ref{prop:full_degenerescence} uses almost nothing and is given in Subsection~\ref{subsec:proof_sticking}. Except for that, the whole section is dedicated to the proof of Theorem \ref{thm:traj_reg}. For this we take once for all $A,P,Q \in \R^N$ and $T>0$, $A$ being strictly ordered and $P,Q$ being ordered. 

Even if all the arguments are elementary, we will need a certain number of steps, including the explicit computation of the potential $|X - \overline \nabla f(X)|^2$ (Subsection~\ref{subsec:lemma_gradient} and \ref{subsec:decomposition_potential}) and the justification of \emph{a priori} knowledge on the optimal trajectories: they can be supposed to be ordered at all time (Subsection~\ref{subsec:order}), and the conservation of energy and momentum holds during shocks\footnote{We say that $\ZZ$ presents a shock at time $t$ if $t$ is a discontinuity point of $\pi(\ZZ)$, see Definition~\ref{def:shock}.} (Subsection~\ref{subsec:conserved_quantities}). The main ingredient in the proof of Theorem~\ref{thm:traj_reg} is an estimate given in Subsection~\ref{subsec:minimal_deviation}: during a non-pathological shock (pathological shocks are excluded \emph{a posteriori}), at least one particle has a below-bounded jump in its velocity (Proposition~\ref{prop:reg_isolated_shock}). We finally provide the proof of Theorem~\ref{thm:traj_reg} in Subsection~\ref{subsec:proof_regularity}.

Throughout the section, we will work with several type of finite sets: the partitions of type $\pi(X)$ and the class of particles of type $C(X,i)$. Some of the arguments or computations will deal with their cardinal. Thus, if $\mathcal{F}$ is a finite set, we will denote by $\#\mathcal{F}$ its cardinal.

\subsection{Properties of the extended gradient}
\label{subsec:lemma_gradient}
The extended gradient of $f$ can be computed explicitly In Lemma \ref{lem:prop_nablaf}, we gather easy properties of $\overline{\nabla}f$ that will be needed in the following. Before doing so, let us introduce some notations.
\begin{Def}
Let $\pi$ be a partition of $\llbracket1,N \rrbracket$. We call $E_\pi$ the linear subspace of $\R^N$ of all $X$ s.t. $\pi$ is a refinement of $\pi(X)$, that is:
\begin{equation*}
E_\pi := \bigcap_{C \in \pi} \bigcap_{i,j \in C} \big\{ X = (x_1,\dots, x_N) \in \R^N \, \big| \, x_i = x_j \big\}.
\end{equation*}
\end{Def}
Here is the lemma:
\begin{Lem}[Properties of $\overline{\nabla}f$]
\label{lem:prop_nablaf}
\begin{enumerate}
\item The extended gradient $\overline{\nabla} f$ has the following symmetry:
\begin{equation}
\label{eq:symmetry_gradient}
\forall X \in \R^N, \, \forall \sigma \in \perm, \qquad \overline{\nabla} f(X^\sigma) = \big( \overline{\nabla} f(X)\big)^{\sigma}.\phantom{\forall X \in \R^N, \, \forall \sigma \in \perm, \qquad}
\end{equation} 
\item The function $X \mapsto |X - \overline{\nabla} f(X)|$ is symmetric:
\begin{equation}
\label{eq:symmetry_potential}
\forall X \in \R^N, \, \forall \sigma \in \perm, \qquad |X^\sigma - \overline{\nabla} f(X^\sigma)|^2 = |X - \overline{\nabla}f(X)|^2. \phantom{\forall X \in \R^N, \, \forall \sigma \in \perm, \qquad}
\end{equation}
\item If $X$ is ordered, then $\overline{\nabla}f(X)$ is the orthogonal projection of $A$ on $E_{\pi(X)}$.
\item If $X$ is ordered and $i \in \{1, \dots, N\}$:
\begin{equation}
\label{eq:coordinate_nablaf}
\forall i= 1, \dots, N, \qquad \big(\overline{\nabla} f(X)\big)_i = \frac{1}{\# C(X,i)} \sum_{j \in C(X,i)} a_j. \phantom{\forall i= 1, \dots, N, \qquad}
\end{equation}
(Recall that $C(X,i)$ is defined in Definition \ref{def:partition}.)
\end{enumerate}
\end{Lem}
\begin{Rem}
The extended gradient $\overline{\nabla} f$ is completely characterized by points 1. and 3. (or 4.) of Lemma~\ref{lem:prop_nablaf}.
\end{Rem}
\begin{proof}
\underline{Point 1}. Let $\sigma \in \perm$. By the definition~\eqref{eq:def_f_1d} of $f$, for all $X \in \R^N$, $f(X^\sigma) = f(X)$. Calling $I^\sigma: X \mapsto X^\sigma$, we easily deduce that at the level of subdifferentials: $\partial^- f(X^\sigma) = I^\sigma \big( \partial^- f(X) \big)$. We conclude by the fact that $I^\sigma$ is orthogonal.

\noindent \underline{Point 2}. It is a direct consequence of Point 1.

\noindent \underline{Point 3}. Let $X=(x_1, \dots x_N) \in \R^N$ be an ordered vector. Considering the definition \eqref{eq:def_f_1d} of $f$ and noticing that the maximum is achieved exactly for those $\sigma$ such that $X^\sigma = X$, it appears that $\overline{\nabla} f(X)$ belongs to the convex hull:
\begin{equation*}
\mathrm{Conv}\Big( \big\{ A^\sigma\, \big| \, \sigma \in \perm \mbox{ such that } X^\sigma = X \big\} \Big).
\end{equation*}

For a given $i \in \{1, \dots, N\}$, we call $V^i \in \R^N$ the vector whose $j$-th coordinate is $1$ if $j \in C(X,i)$ and $0$ otherwise. On the one hand, we have $E_{\pi(X)} = \mathrm{Span}\{ V^i\, | \, i=1, \dots, N \}$, and on the other hand, for all $i$, the scalar product $V^i \cdot Y$ is constant on the above-mentioned convex hull. So we deduce: 
\begin{equation*}
A - \overline{\nabla}f(X) \in \big(E_{\pi(X)}\big)^\perp.
\end{equation*}
Hence, we just have to prove that $\overline{\nabla}f(X) \in E_{\pi(X)}$. If $i,j \in \{1, \dots, N\}$ are such that $x_i = x_j$, let us apply formula~\eqref{eq:symmetry_gradient} to the permutation $\sigma := (i,j)$:
\begin{equation*}
\big(\overline{\nabla}f(X)\big)_i=\big(\big(\overline{\nabla}f(X)\big)^\sigma \big)_j = \big(\overline{\nabla}f(X^\sigma) \big)_j = \big(\overline{\nabla}f(X) \big)_j.
\end{equation*}
The result follows.

\noindent \underline{Point 4}. Let $X$ be ordered and $i \in \{1, \dots, N\}$. As $\overline{\nabla} f(X) \in E_{\pi(X)}$, with the notations of the proof of Point~3:
\begin{align*}
\big(\overline{\nabla} f(X)\big)_i &= \frac{1}{\#C(X,i)} \sum_{j\in C(X,i)} \big(\overline{\nabla} f(X)\big)_j = \frac{1}{\#C(X,i)} \overline{\nabla} f(X) \cdot V^i \\
&= \frac{1}{\#C(X,i)} A \cdot V^i = \frac{1}{\#C(X,i)} \sum_{j\in C(X,i)} a_j,
\end{align*}
where we used $A - \overline{\nabla}f(X) \perp V^i$ to get the first identity in the second line.
\end{proof}
The three next subsections will be dedicated to consequences of this lemma:
\begin{itemize}
\item A proof of Proposition~\ref{prop:full_degenerescence};
\item When proving Theorem \ref{thm:traj_reg}, it is enough to consider ordered trajectories (Proposition~\ref{prop:ordered}); 
\item For ordered trajectories, the potential in $\Lambda'$ can be decomposed as sum of a smooth "external" potential and an "internal" energy only depending on $\pi(X)$ (Proposition~\ref{prop:decomposition}).
\end{itemize}

\subsection{Proof of Proposition \ref{prop:full_degenerescence}}
\label{subsec:proof_sticking}
With the help of Lemma \ref{lem:prop_nablaf}, we are ready to prove Proposition \ref{prop:full_degenerescence}.
\begin{proof}[Proof of Proposition \ref{prop:full_degenerescence}]
\underline{Point 1}. Without loss of generality, we can suppose $t_1=0$ and $t_2=T$, that is $P=(p_1, \dots, p_N)$ and $Q=(q_1, \dots, q_N)$ are such that $p_1 = \dots = p_N$ and $q_1 = \dots = q_N$.
	
	Call $\Psi$ the orthogonal projection on the line $E_{\llbracket 1,N \rrbracket} := \{X=(x_, \dots, x_N) \in \R^N \, | \, x_1 = \dots= x_N\}$. It suffices to prove that when $\ZZ$ is a continuous trajectory joining $P$ to $Q$, then $\Lambda'(\Psi(\ZZ)) \leq \Lambda'(\ZZ)$, and with equality if and only if $\ZZ = \Psi(\ZZ)$. As $\Psi$ is $1$-Lipschitz, it reduces the kinetic part of $\Lambda'$. For the potential part, we remark that for all $X \in \R^N$, $E_{\pi(\Psi(X))} = E_{\llbracket 1, N\rrbracket} \subset E_{\pi(X)}$. As a consequence, by Point 3. of Lemma~\ref{lem:prop_nablaf}, we have as soon as $X$ is ordered $\overline{\nabla} f (\Psi(X)) = \Psi(\overline{\nabla} f(X))$. Hence:
	\begin{equation*}
	|\Psi(X) - \overline{\nabla} f(\Psi(X))|^2 = \big| \Psi\big(X - \overline\nabla f(X)  \big)\big|^2 \leq |X - \overline\nabla f (X)|^2,
	\end{equation*}
	with equality if and only if $X \in E_{\llbracket 1, N \rrbracket}$, \emph{i.e.} if and only if $\Psi(X) = X$. This property is extended to non-ordered $X$ using \eqref{eq:symmetry_potential}, and the result follows.
	
\noindent \underline{Point 2}. The function $\overline{\Lambda'} = \overline{\Lambda'}(P,Q)$ defined for all $P,Q \in \R^N$ as the minimal value of $\Lambda'$ is continuous. Indeed, if $P,P',Q,Q' \in \R^N$ are chosen so that $|P'-P| + |Q'-Q| \ll 1$ and if $\ZZ$ is a trajectory joining $P$ to $Q$, we can find a trajectory $\widetilde{\ZZ}$ joining $P'$ to $Q'$ with:\footnote{With a slight abuse of notation, we do not refer explicitly to the dependence of $\Lambda'$ on $P,Q$.}
\begin{equation}
\label{eq:lsc}
\Lambda'(\widetilde \ZZ) \leq \Lambda'(\ZZ) + \underset{(P',Q') \to (P,Q)}{o}(1).
\end{equation} 
To do so, it suffices to choose $\tau \sim |P'-P| + |Q'-Q|$, and to define $\widetilde{\ZZ}$ as the trajectory joining $P'$ to $P$ in straight line between times $0$ and $\tau$, joining $P$ to $Q$ between times $\tau$ and $T-\tau$ by following $\ZZ$ with a proper affine change of time, and finally joining $Q$ to $Q'$ in straight line between times $T-\tau$ and $T$. This shows that $\overline{\Lambda'}$ is lower semi-continuous, but the continuity is obtained by noticing that the $o$ in \eqref{eq:lsc} is locally uniform on $P,Q \in \R^N$. The argument is easily adapted to show that $\widetilde{\Lambda'} = \widetilde{\Lambda'}(P,Q)$ defined for $P,Q \in \R^N$ by:
\begin{equation*}
\widetilde{\Lambda'}(P,Q) := \inf \big\{ \Lambda'(\ZZ) \, \big| \, \ZZ \mbox{ whose set of } t \mbox{ s.t. } Z_t \in E_{\llbracket 1, N \rrbracket} \mbox{ is negligible}  \big\}
\end{equation*}
is also continuous. Besides, the set $\mathcal{U}$ defined in the statement clearly satisfies:
\begin{equation*}
\mathcal{V} := \big\{ P,Q \in \R^N \, \big| \, \overline{\Lambda'}(P,Q) < \widetilde{\Lambda'}(P,Q) \big\} \subset \mathcal{U}.
\end{equation*}
By continuity of $\overline{\Lambda'}$ and $\widetilde{\Lambda'}$, $\mathcal{V}$ is an open set. Hence it remains to prove that:
\begin{equation*}
\big\{ P,Q \in \R^N \, \big| \, p_1 = \dots = p_N \mbox{ and } q_1 = \dots =q_N  \big\} = E_{\llbracket 1, \N \rrbracket} \times E_{\llbracket 1, \N \rrbracket} \subset \mathcal{V}.
\end{equation*}
To do so, we take $P,Q \in E_{\llbracket 1, N \rrbracket}$, $\ZZ$ a curve joining $P$ to $Q$ such that $\{t \, | \, Z_t \in E_{\llbracket 1, N \rrbracket}\}$ is negligible, we still call $\Psi$ the orthogonal projection on $E_{\llbracket 1, N \rrbracket}$, and we prove that
\begin{equation*}
\Lambda'(\ZZ) \geq \Lambda'(\Psi(\ZZ)) + a,
\end{equation*}
where $a>0$ does not depend on $\ZZ$. Let us call $\Phi := \mathrm{Id} - \Psi$ the orthogonal projection on the orthogonal of $E_{\llbracket 1, N \rrbracket}$. As in the proof of the first point, $\overline\nabla f\circ \Psi= \Psi \circ \overline\nabla f$. As a consequence:
\begin{align*}
\Lambda'(Z) &= \int_0^T\left\{ |\Psi(\dot Z_t)|^2 + |\Psi(Z_t) - \Psi(\overline\nabla f(Z_t))|^2 \right\}\D t + \int_0^T\left\{ |\Phi(\dot Z_t)|^2 + |\Phi(Z_t) - \Phi(\overline\nabla f(Z_t))|^2 \right\}\D t \\
&= \Lambda'(\Psi(\ZZ)) + \int_0^T\left\{ |\dot Z^\perp_t|^2 + |Z^\perp_t - \Phi(\overline\nabla f(Z_t))|^2 \right\}\D t,
\end{align*}
where $\ZZ^\perp = Z^\perp_t:= \Phi(Z_t)$ is a curve joining $0$ to $0$. But for almost all $t$, $Z_t \notin E_{\llbracket 1, N \rrbracket}$, so as we saw in the proof of the first point, $\overline\nabla f(Z_t) \notin E_{\llbracket 1, N \rrbracket}$. As $\overline\nabla f$ only takes a finite number of values (see Lemma~\ref{lem:prop_nablaf}), for almost all $t$, $\Phi(\overline \nabla f(Z_t))$ belongs to some finite set, say $\mathcal{G}$, which does not contain $0$. Hence,
\begin{equation*}
 \int_0^T\left\{ |\dot Z^\perp_t|^2 + |Z^\perp_t - \Phi(\overline\nabla f(Z_t))|^2 \right\}\D t \geq \int_0^T\left\{ |\dot Z^\perp_t|^2 + \mathrm{dist}(Z^\perp_t, \mathcal{G})^2 \right\}\D t,
\end{equation*}
where $\mathrm{dist}(Z, \mathcal{G})$ denotes the distance from $Z$ to $\mathcal{G}$. Because $\ZZ^\perp$ joins $0$ to $0$ and $\mathcal{G}$ does not contain $0$, this last integral is easily seen to be below bounded away from $0$ independently of $\ZZ$, and the result follows.
\end{proof}

\subsection{Ordering of the particles}
\label{subsec:order}
The purpose of this subsection is to show that when proving Theorem \ref{thm:traj_reg}, we can restrict ourselves to study trajectories that remain ordered (see Figure~\ref{fig:ordered}). This is due to the following proposition.
\begin{figure}
\hspace{2cm}
\begin{tikzpicture}
\draw[thick, ->] (0,-0.3) -- (0,4.6);
\draw[thick, densely dashed, blue] (1,0) to[bend left] (1.6, 0.5);
\draw[thick, densely dashed, blue] (1.8,1) to[out = 20, in = -160] (2.6, 2);
\draw[very thick, dotted, ForestGreen] (2.9,2.8) to[bend right = 15] (1.8,3.5);
\draw[very thick, dotted, ForestGreen] (1.785,3.5) to[bend left = 15] (1.685,3.8);
\draw[thick, red] (1.815,3.5) to[bend left = 15] (1.715,3.8);
\draw[very thick, dotted, ForestGreen] (1.7,3.8) to[bend left = 10] (0.7,4.3);
\draw[thick, densely dashed, blue] (1.585,0.5) to[bend left = 10] (1.785,1);
\draw[thick, densely dashed, blue] (2.585,2) to[bend right = 10] (2.885,2.8);
\draw[thick, red] (2,0) to[bend right] (1.6, 0.5);
\draw[thick, red] (1.8,1) to[bend left] (1.8,3.5);
\draw[thick, red] (1.7,3.8) to[bend right = 10] (2.5,4.3);
\draw[thick, red] (1.615,0.5) to[bend left = 10] (1.815,1);
\draw[very thick, dotted, ForestGreen] (4,0) to[bend right] (2.6,2);
\draw[thick, densely dashed, blue] (2.9,2.8) to[bend right = 20] (3.6,4.3);
\draw[very thick, dotted, ForestGreen] (2.615,2) to[bend right = 10] (2.915,2.8);
\draw (2pt,0) -- (-2pt,0) node[above left]{$0$};
\draw (2pt,4.3) -- (-2pt, 4.3) node[above left]{$T$};
\draw plot[mark = +] (1,0) node[below left]{$p_1$};
\draw plot[mark = +] (2,0) node[below right]{$p_2$};
\draw plot[mark = +] (4,0) node[below right]{$p_3$};
\draw plot[mark = +] (0.7,4.3) node[above left]{$q_1$};
\draw plot[mark = +] (2.5,4.3) node[above right]{$q_2$};
\draw plot[mark = +] (3.6,4.3) node[above right]{$q_3$};
\end{tikzpicture}
\hfill
\begin{tikzpicture}
\draw[thick, ->] (0,-0.3) -- (0,4.6);
\draw[thick, densely dashed, blue] (1,0) to[bend left] (1.6, 0.5);
\draw[thick, red] (1.8,1) to[out = 20, in = -160] (2.6, 2);
\draw[thick, red] (2.9,2.8) to[bend right = 15] (1.8,3.5);
\draw[thick, densely dashed, blue] (1.785,3.5) to[bend left = 15] (1.685,3.8);
\draw[thick, red] (1.815,3.5) to[bend left = 15] (1.715,3.8);
\draw[thick, densely dashed, blue] (1.7,3.8) to[bend left = 10] (0.7,4.3);
\draw[thick, densely dashed, blue] (1.585,0.5) to[bend left = 10] (1.785,1);
\draw[thick, red] (2.585,2) to[bend right = 10] (2.885,2.8);
\draw[thick, red] (2,0) to[bend right] (1.6, 0.5);
\draw[thick, densely dashed, blue] (1.8,1) to[bend left] (1.8,3.5);
\draw[thick, red] (1.7,3.8) to[bend right = 10] (2.5,4.3);
\draw[thick, red] (1.615,0.5) to[bend left = 10] (1.815,1);
\draw[very thick, dotted, ForestGreen] (4,0) to[bend right] (2.6,2);
\draw[very thick, dotted, ForestGreen] (2.9,2.8) to[bend right = 20] (3.6,4.3);
\draw[very thick, dotted, ForestGreen] (2.615,2) to[bend right = 10] (2.915,2.8);
\draw (2pt,0) -- (-2pt,0) node[above left]{$0$};
\draw (2pt,4.3) -- (-2pt, 4.3) node[above left]{$T$};
\draw plot[mark = +] (1,0) node[below left]{$p_1$};
\draw plot[mark = +] (2,0) node[below right]{$p_2$};
\draw plot[mark = +] (4,0) node[below right]{$p_3$};
\draw plot[mark = +] (0.7,4.3) node[above left]{$q_1$};
\draw plot[mark = +] (2.5,4.3) node[above right]{$q_2$};
\draw plot[mark = +] (3.6,4.3) node[above right]{$q_3$};
\end{tikzpicture}
\hspace{2cm}
\caption{\label{fig:ordered} These two trajectories share their initial and final position up to ordering and their actions. But to the right, the order is preserved while to the left, this is not the case.}
\end{figure}
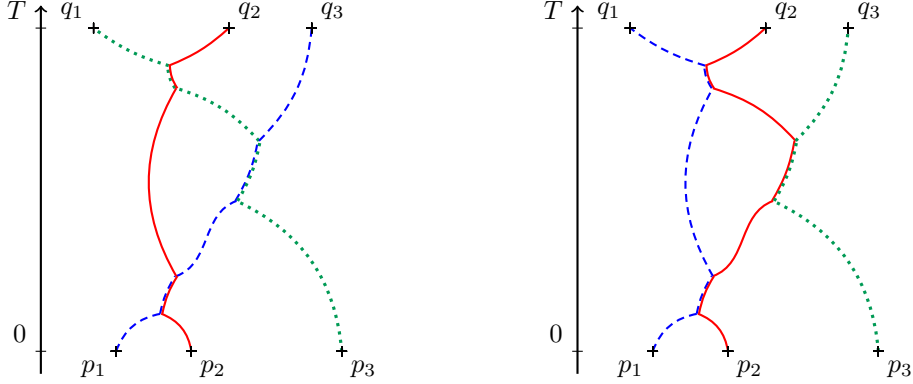
\begin{Prop}
\label{prop:ordered}
Let $\ZZ = Z_t$ be a global minimizer of $\Lambda'$. We call $\widetilde{\ZZ} = \widetilde{Z}_t$ the trajectory obtained by reordering the coordinates of $\ZZ$ in increasing order. Then $\widetilde{\ZZ}$ is also a global minimizer of $\Lambda'$. 

Moreover, $\ZZ$ has the regularity stated in Theorem \ref{thm:traj_reg} if and only if $\widetilde{\ZZ}$ does.

In particular, $\Lambda'$ always admits an ordered minimizer, and it it is enough to prove Theorem \ref{thm:traj_reg} for such minimizers.
\end{Prop} 
Thanks to this proposition, from now on, we only work with ordered minimizers of $\Lambda'$. These minimizers $\ZZ = Z_t$ satisfy in particular $Z_0=P$ and 
$Z_T = Q$ (as we chose them to be ordered in the first place).
\begin{proof}
Let $\ZZ$ and $\widetilde{\ZZ}$ be as in the statement of the proposition. Point 2 of Lemma \ref{lem:prop_nablaf} implies:
\begin{equation*}
\int_{0}^{T} |\widetilde{Z}_t - \overline{\nabla} f( \widetilde{Z}_t) |^2  \D t= \int_{0}^{T} |Z_t - \overline{\nabla} f( Z_t) |^2 \D t .
\end{equation*}

We call $\Psi: \R^N \to \R^N$ the operator that reorders the coordinates of a vector in increasing order, so that in particular for all $t$, $\widetilde{Z}_t = \Psi(Z_t)$. A simple application of the rearrangement inequality shows that $\Psi$ is $1$-Lipschitz. In particular, it reduces the action of curves:
\begin{equation*}
\int_{0}^{T} | \dot{\widetilde{Z}}_t |^2 \D t  \leq \int_{0}^{T} | \dot{Z}_t |^2 \D t.
\end{equation*}
By adding the two last formulas, and by noticing that the endpoint constraint is fulfilled, we get $\Lambda'(\widetilde{\ZZ}) \leq \Lambda'(\ZZ)$. As $\ZZ$ is a minimizer, this inequality is in fact an equality, and $\widetilde{\ZZ}$ is also a minimizer. 

Remark that both $\ZZ$ and $\widetilde{\ZZ}$ are continuous because they have finite action. Hence, the second claim of the proposition is a consequence of the two following facts:
\begin{itemize}
\item For all $t \in [0,T]$, $\#\pi(\widetilde{Z}_t) = \#\pi(Z_t)$.
\item For any continuous trajectory $t \in I \mapsto X_t \in \R^N$ where $I$ is an interval, $t\mapsto\pi(X_t)$ is constant if and only if $t \mapsto\# \pi(X_t)$ is constant.
\end{itemize}
Indeed in that case, $t \mapsto \pi(Z_t)$ and $t \mapsto\pi(\widetilde{Z}_t)$ are constant on the same intervals, and the result follows.

The first point and the "only if" part of the second point are trivial. 

For the "if" part of the second one, we reason by contraposition. Suppose $s \mapsto \pi(X_s)$ has a discontinuity at time $t$ and we prove that $s \mapsto \# \pi(X_s)$ also does. If $s \mapsto \pi(X_s)$ has a discontinuity at time $t$, we can find two distinct accumulation points $\pi_1$ and $\pi_2$ of $s \mapsto \pi(X_s)$ at time $t$. As for all $\pi$, the set $E_\pi$ is closed, $X_t$ belongs to $E_{\pi_1} \cap E_{\pi_2}$. But this set is noting but $E_{\overline{\pi}}$ where $\overline{\pi}$ is the finest partition of which $\pi_1$ and $\pi_2$ are refinements, that is the partition corresponding to the relation:
\begin{equation*}
i \sim j \quad \Leftrightarrow \quad \exists C \in \pi_1 \cup \pi_2 \, \mbox{  s.t. }\, \{ i,j \} \subset C.
\end{equation*}
In particular, $\pi(X_t)$ is a refinement of $\overline{\pi}$ and as $\pi_1 \neq \pi_2$, we easily get:
\begin{equation*}
\# \pi(X_t) \leq \# \overline{\pi} < \max\big( \# \pi_1, \#\pi_2 \big).
\end{equation*}
So $s \mapsto \#\pi(X_s)$ has a discontinuity at time $t$, and the result follows.
\end{proof}
\subsection{Decomposition of the potential}
\label{subsec:decomposition_potential}
Here, we compute explicitly the values of the potential $X \mapsto |X - \overline{\nabla} f(X)|^2$ on ordered vectors $X \in \R^N$. Notice that for such vectors $X$, $\pi(X)$ has an additional structure: if $C \in \pi(X)$, then $C$ is an interval of integers. We say that such partitions are \emph{ordered}. We prove the following:
\begin{Prop}
\label{prop:decomposition}
For all ordered $X \in \R^N$:
\begin{equation}
\label{eq:decomposition_potential}
|X - \overline{\nabla} f(X)|^2 = |X-A|^2 + h(\pi(X)) - |A|^2,
\end{equation}
where $h$ is defined on a partition $\pi$ of $\llbracket 1, N \rrbracket$ by:
\begin{equation}
\label{eq:def_h}
h(\pi) := \sum_{C \in \pi} \frac{1}{\#C} \Big|\sum_{j \in C} a_j \Big|^2.
\end{equation}
In particular, $h$ has the following monotonicity property: if $\pi$ and $\pi'$ are two ordered partitions and if $\pi'$ is a strict refinement of $\pi$, then $h(\pi) < h(\pi')$.
\end{Prop}

The more particle are stuck together, the lower $h$. This is the reason for which $\Lambda'$ favours the sticking of particles. The function $-h$ can be understood as the internal energy of the system.

Dropping the constant term $|A|^2/2$ in~\eqref{eq:decomposition_potential} and defining $\Lambda''$ on a trajectory $\ZZ$ by:
\begin{equation}
\label{eq:decomposed_functional_dim1}
\Lambda'' ( \ZZ ) = \left\{
\begin{aligned}
&\int_{0}^{T}\left\{ | \dot{Z}_t |^2 + |Z_t - A |^2 + h(\pi(Z_t))\right\} \D t , && \mbox{if } \ZZ \in H^1([0, T]; \R^{N}),\\[-10pt]
&&& Z_{0} =P \mbox{ and } Z_{T} =Q ,\\[5pt]
&+\infty, && \mbox{else,}
\end{aligned}
\right.
\end{equation} 
it is clear that $\Lambda'$ and $\Lambda ''$ have the same minimizers in the class of ordered trajectories. Hence, as a consequence of Proposition~\ref{prop:ordered}, it suffices to prove the conclusion of Theorem~\ref{thm:traj_reg} for the minimizers of $\Lambda''$ in the class of ordered trajectories.
\begin{proof}[Proof of Proposition~\ref{prop:decomposition}]
Let $X \in \R^N$ be an ordered vector. By Point 3. of Lemma \ref{lem:prop_nablaf}, we have $A - \overline{\nabla} f(X) \in \big( E_{\pi(X)}\big)^\perp$ and both $X$ and $\overline{\nabla} f(X) \in E_{\pi(X)}$. So using twice the Pythagorean theorem, we get:

\begin{equation*}
|X-\overline{\nabla} f(X)|^2 = |X - A|^2 - |A - \overline{\nabla} f(X)|^2 = |X - A|^2 + |\overline{\nabla} f(X)|^2 - |A|^2.
\end{equation*}
The identities \eqref{eq:decomposition_potential} and \eqref{eq:def_h} are obtained by computing $|\overline{\nabla} f(X)|^2$ using \eqref{eq:coordinate_nablaf}. 

If we recap, $h(\pi)$ is the squared norm of the orthogonal projection of $A$ on $E_{\pi}$. But if $\pi'$ is a refinement of $\pi$, $E_{\pi} \subset E_{\pi'}$, and hence $h(\pi) \leq h(\pi')$. The strict inequality is obtained by noticing with the help of \eqref{eq:coordinate_nablaf} and using the strict ordering of $A$ that if in addition $\pi$ and $\pi'$ are ordered and $\pi' \neq \pi$, then the projection of $A$ on $E_{\pi'}$ does not belong to $E_\pi$. 
\end{proof}
\subsection{Conserved quantities}
\label{subsec:conserved_quantities}
In this subsection, we discuss two simple and yet structural properties of the dynamic prescribed by the functionals $\Lambda'$, $\Lambda''$: the Hamiltonian of the system is conserved (Proposition \ref{prop:conservation_energy}), and its center of mass is smooth (Proposition \ref{prop:conservation_impulsion}). In particular, the momentum of the system is conserved during shocks.
\begin{Prop}
\label{prop:conservation_energy}
Let $\ZZ$ be an ordered minimizer of $\Lambda''$. Then:
\begin{equation}
\label{eq:def_energy}
\mathcal{E} = \mathcal{E}(t) := |\dot Z_t|^2 - |Z_t - A|^2 - h(\pi(Z_t))
\end{equation}
is constant in the sense of distributions.
\end{Prop}
\begin{proof}
The proof is completely standard and consists in comparing the value of $\Lambda''$ on $\ZZ$ and $t \mapsto Z_{t + \eps \varphi(t)}$ for small $\eps$ and functions $\varphi$ that are smooth and compactly supported in $(0,T)$.
\end{proof}
\begin{Prop}
\label{prop:conservation_impulsion}
Let $\ZZ = (z_1(t), \dots, z_N(t))$ be an ordered minimizer of $\Lambda''$. Call $a := (a_1 + \dots + a_N)/N$ and for $t \in [0,T]$:
\begin{equation*}
\mathcal{M}(t) := \frac{1}{N} \sum_{i=1}^N z_i(t)\qquad \mbox{and} \qquad \mathcal{P}(t) := \frac{1}{N} \sum_{i=1}^N \dot z_i(t).
\end{equation*}
($\mathcal{M}$ is well define for all $t$, and $\mathcal{P}$ for almost all $t$.) Then $\mathcal{M}, \mathcal{P}$ solve distributionally:
\begin{gather*}
\dot{\mathcal{M}}(t) = \mathcal{P}(t),\\
\dot{\mathcal{P}}(t) = \mathcal{M}(t) - a. 
\end{gather*}
In particular, $\mathcal{M}$ is smooth and $\mathcal{P}$ coincide almost surely with a smooth function.
\end{Prop}
\begin{proof}
Here the proof consists in comparing the value of $\Lambda''$ on $\ZZ$ and $t \mapsto Z_t + \eps \varphi(t) V$ for small $\eps$, smooth and compactly supported $\varphi$, and where we call $V = (1, \dots, 1)$. The only somehow unusual thing to remark is that $\pi$ and hence $h \circ \pi$ are invariant under translations in the direction of $V$.
\end{proof}

\subsection{Shock, isolated shocks and minimal deviation}
\label{subsec:minimal_deviation}
This subsection contains the main estimate that allows to prove Theorem \ref{thm:traj_reg}. Roughly speaking, if at time $t$ some of the particles stick or separate, there is a below bound on the change of the velocity of at least one particle. The proof of Theorem \ref{thm:traj_reg} will then consist in showing that this cannot happen an infinite number of time.

Let us first define as "shocks" these sticking and separating behaviours:
\begin{Def}[Shocks]
\label{def:shock}
Let $\XX = X_t = (x_1(t), \dots , x_N(t))$ be a continuous trajectory on $\R^N$.
\begin{enumerate}
\item We call a shock of $\XX$ a triplet $(t,q,C)$ with $t \in [0,T]$, $q \in \R$ and $C \subset\llbracket 1, N \rrbracket$ such that:
\begin{itemize}
\item $C \in \pi(X_t)$,
\item for all $i \in C$, $x_i(t) = q$,
\item for all $\tau>0$, there exists $s \in (t-\tau, t+\tau)$ such that $C\notin \pi(X_s)$.
\end{itemize}

\item If $(t,q,C)$ is a shock of $\XX$, we say that it is isolated if $(t,q)$ is isolated in:
\begin{equation*}
\big\{ (t',q') \big| \, \exists C' \subset \llbracket1, N\rrbracket \mbox{ s.t. } (t',q',C') \mbox{ is a shock} \big\},
\end{equation*}
\emph{i.e.} if there is no other shock than $(t,q,C)$ in the neighbourhood of $(t,q)\in [0,T]\times \R$.
\end{enumerate}
\end{Def}
We provide in Figure \ref{fig:nonisolated_shock} a picture of a shock which does not seem to be isolated.
\begin{figure}
\centering
\begin{tikzpicture}
\draw[thick, ->] (0,-0.3) -- (0,2.8);
\draw[thick, densely dashed, blue] (2,0) to[bend right=10] (1.835,0.4);
\draw[thick, densely dashed, blue] (1.835,0.4) to[bend left=10] (1.735,0.6);
\draw[thick, densely dashed, blue] (1.735,0.6) to[bend left] (1.285,1.4);
\draw[thick, densely dashed, blue] (1.285,1.4) to[bend left=10] (0.985,1.8);
\draw[thick, densely dashed, blue] (0.985,1.8) to[bend left] (0.5,2.5);
\draw[very thick, dotted, ForestGreen] (2,0) to[bend left=10] (2.065,0.25);
\draw[very thick, dotted, ForestGreen] (2.065,0.25) to[bend right=10] (2.095,0.35);
\draw[very thick, dotted, ForestGreen] (2.095,0.35) to[bend right] (2.215,0.9);
\draw[very thick, dotted, ForestGreen] (2.215,0.9) to[bend right=10] (2.315,1.15);
\draw[very thick, dotted, ForestGreen] (2.315,1.15) to[bend right] (3,2.5);
\draw[thick, red] (2.035,0.25) to[bend right=10] (2.065,0.35);
\draw[thick, red] (2.065,0.35) -- (1.865,0.4);
\draw[thick, red] (1.865,0.4) to[bend left=10] (1.765,0.6);
\draw[thick, red] (1.765,0.6) -- (2.185,0.9);
\draw[thick, red] (2.185,0.9) to[bend right=10] (2.285,1.15);
\draw[thick, red] (2.285,1.15) -- (1.315,1.4);
\draw[thick, red] (1.315,1.4) to[bend left=10] (1.015,1.8);
\draw[thick, red] (1.015,1.8) to[bend right=15] (1.8,2.5);
\draw[thick, red] (1.955,0.22) -- (2.035,0.25);
\draw[thick, red] (1.955,0.22) -- (1.98,0.17);
\draw[thick, red] (1.98,0.17) -- (2.01, 0.15);
\draw[thick, red] (2.01,0.15) to[bend right] (2,0);
\draw plot[mark = +] (2,0) node[below]{$p_1 = p_2 = p_3$};
\draw plot[mark = +] (0.5,2.5) node[above]{$q_1$};
\draw plot[mark = +] (1.8,2.5) node[above]{$q_2$};
\draw plot[mark = +] (3,2.5) node[above]{$q_3$};
\draw (2pt,0) -- (-2pt,0) node[above left]{$0$};
\draw (2pt,2.5) -- (-2pt, 2.5) node[above left]{$T$};
\end{tikzpicture}
\caption{\label{fig:nonisolated_shock} A shock with three particles which does not seem to be isolated. We will see later on that this kind of shock cannot occur in our model.}
\end{figure}
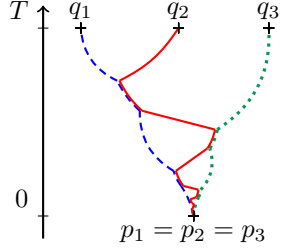
The following result is the main step in the proof of Theorem \ref{thm:traj_reg}.
\begin{Prop}
\label{prop:reg_isolated_shock}
Let $\ZZ = (z_1(t), \dots, z_N(t))$ be an ordered minimizer of $\Lambda'$ (or equivalently a minimizer of $\Lambda''$ in the class of ordered trajectories), and let $t \in [0,T]$.
\begin{enumerate}
\item If particle $i$ is not involved in a shock at time $t$, then for $s$ in the neighbourhood of $t$, $C := C(Z_s,i)$ is constant and $z_i$ is a smooth solution of:
\begin{equation}
\label{eq:Newton_law}
\ddot z_i(s) = z_i(s) - \frac{1}{\#C} \sum_{j \in C} a_j.
\end{equation}
In particular, if $i$ is involved in an isolated shock at time $t$, then $z_i$ admits left and right derivatives at time $t$, denoted by $\dot z_i(t-)$ and $\dot z_i(t+)$ respectively.
\item There is $\alpha = \alpha(N,A)>0$ such that for any isolated shock $(t,q,C)$, calling $i := \min C$:
\begin{equation}
\label{eq:minimal_deviation}
\dot z_i(t-) - \dot z_i(t+) \geq \alpha.
\end{equation}
\end{enumerate}
\end{Prop}
\begin{proof}
\underline{Point 1}. If particle $i$ is not involved in a shock at time $t$, by definition of a shock, it means that $C := C(Z_t,i) \in \pi(Z_s)$ for all $s$ in a neighbourhood of $t$. In particular, for all $j \in C$ and $s$ sufficiently close to~$t$, by~\eqref{eq:coordinate_nablaf}:
\begin{equation*}
 \big(\overline{\nabla} f(Z_s)\big)_j = \frac{1}{\# C} \sum_{k \in C} a_k. 
\end{equation*}
On the other hand, it is easy to find a neighbourhood $U$ of $(t,z_i(t))$ in $[0,T] \times \R$ such that for all $j \in \{1, \dots, N \}$ and all $s\in[0,T]$, $(s, z_j(s)) \in U$ implies $j \in C$.  

As a consequence, if $\xi:[0,T] \to \R$  is smooth and compactly supported in a sufficiently small neighbourhood of $t$, and if $\eps$ is sufficiently small, by defining $\widetilde{\ZZ} = (\widetilde z_1(s), \dots, \widetilde z_N(s))$ for any $j\in \{1, \dots, N\}$ and $s \in [0,T]$ by:
\begin{equation*}
\widetilde z_j(s) := \left\{\begin{aligned}
&z_j(s) + \eps \xi(s) && \mbox{if } j \in C,\\
&z_j(s) && \mbox{else},\\
\end{aligned} \right.
\end{equation*}
then $\pi(\ZZ)$ and $\pi(\widetilde \ZZ)$ (and hence $\overline{\nabla} f(\ZZ)$ and $\overline{\nabla} f(\widetilde{\ZZ})$) coincide at all time. The ODE follows from comparing the values of $\Lambda'$ on $\ZZ$ and trajectories of type $\widetilde \ZZ$.

In particular, by boundedness of $\ZZ$, if particle $i$ is not involved in a shock at time $t$, $|\ddot z_i|$ is bounded by a constant not depending on $t$. The existence of $\dot z_i(t-)$ and $\dot z_i(t+)$ at the times of isolated shocks follows easily.

\noindent\underline{Point 2}. This is the heart of our study of the dynamical system, and maybe the less standard part of Section~\ref{sec:sticky}. But still the idea is very easy: With the notations of the statement, if $\dot z_i(t-) - \dot z_i(t+)$ is too small, then it is cheaper to stick particle $i$ with other particles, as shown in Figure \ref{fig:testtrajectory}. The proof goes as follows. 

\noindent\underline{Step 1}: Definition of a competitor.

Let us consider $(t,q,C)$ an isolated shock. Because it is isolated, we can find $\tau>0$ such that the particles of $C$ are not involved in an other shock between times $t-\tau$ and $t+\tau$. By definition of a shock, we cannot have $C \in \pi(Z_s)$ for all $s \in (t-\tau, t+\tau)$, so either for all $s \in (t-\tau,t)$, $C \notin \pi(Z_s)$ or for all $s \in (t,t+\tau)$, $C \notin \pi(Z_s)$. Without loss of generality, we suppose that the second one holds: the particles of $C$ are not all stuck right after the shock. Moreover, by our choice of $\tau$, for all $C' \subset C$, the assertion $C' \in \pi(Z_s)$ is either true of false independently on $s \in (t, t+\tau)$. Then, for $s \in (t, t+\tau)$, the following definitions of $C_1,C_2 \in \pi(Z_s)$ do not depend on $s$:
\begin{equation*}
C_1 := C(Z_s,i) \mbox{ for }i = \min C \qquad \mbox{and} \qquad C_2  := C( Z_s,i) \mbox{ for } i = \min C\backslash C_1.
\end{equation*}
(The classes $C_1$ and $C_2$ are the two leftmost packs of particles of $C$ right after the shock.) Let us define for $j = 1,2$:
\begin{equation}
\label{eq:def_kvp}
k_j := \# C_j, \qquad v_j := \dot z_i(t+) \mbox{ for }i \in C_j, \quad \mbox{and} \quad p := \frac{k_1 v_1 + k_2 v_2}{k_1 + k_2}.
\end{equation}
For $0\leq \sigma < \tau$ and $\lambda \in [0,1)$, we define a competitor $\ZZ^{\sigma,\lambda} = (z_1^{\sigma,\lambda}(s), \dots, z_N^{\sigma,\lambda}(s))$ by setting for all $i = \{1, \dots, N\}$ and $s \in [0,T]$:
\begin{equation*}
z_i^{\sigma,\lambda} (s) = \left\{
\begin{aligned}
&z_i(s) && \mbox{if }i \notin C_1\cup C_2 \mbox{ or } s \notin (t, t+\sigma),\\
& q + (s-t)p && \mbox{if } i \in C_1 \cup C_2 \mbox{ and } s \in (t, t+\lambda \sigma),\\
&\frac{t + \sigma - s}{(1-\lambda)\sigma}\big( q + \lambda\sigma p \big) + \frac{s - (t + \lambda \sigma)}{(1-\lambda) \sigma}z_i(t + \sigma) &&\mbox{if } i \in C_1 \cup C_2 \mbox{ and } s \in (t + \lambda \sigma, t + \sigma).
\end{aligned} 
\right.
\end{equation*}
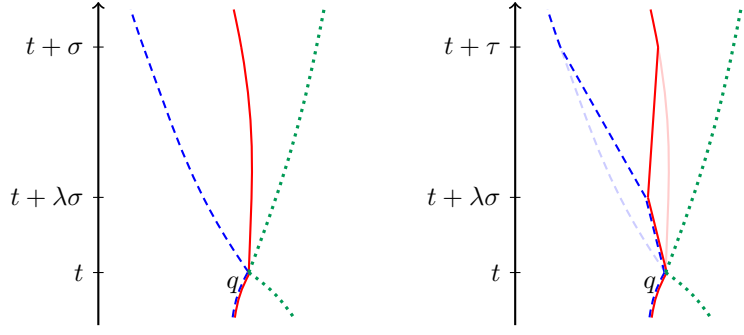
\begin{figure}
\centering
\begin{tikzpicture}
\draw[thick, ->] (0,-0.7) -- (0,3.6);
\draw (2pt,0) -- (-2pt,0) node[left]{$t$};
\draw[thick, densely dashed, blue] (2,0) to[out=125, in=-70] (0.6,3);
\draw[thick, densely dashed, blue] (0.6,3) to[out=110, in=-72] (0.43,3.5) ;
\draw[thick, densely dashed, blue] (1.985,0) to[bend right=10] (1.785,-0.6);
\draw[thick, red] (2,0) to[out=88, in=-80] (1.9,3);
\draw[thick, red] (1.9,3) to[out=100, in=-78] (1.8,3.5);
\draw[thick, red] (2.015,0) to[bend right=10] (1.815,-0.6);
\draw[very thick, dotted, ForestGreen] (2,0) to[bend right=5] (3,3.5);
\draw[very thick, dotted, ForestGreen] (2,0) to[bend left=10] (2.6,-0.6);
\draw (2,0) node[left=6pt, below=-2pt]{$q$};
\draw (2pt,1) -- (-2pt, 1) node[left]{$t + \lambda\sigma$};
\draw (2pt,3) -- (-2pt, 3) node[left]{$t + \sigma$};
\end{tikzpicture}
\hspace{1cm}
\begin{tikzpicture}
\draw[thick, ->] (0,-0.7) -- (0,3.6);
\draw (2pt,0) -- (-2pt,0) node[left]{$t$};
\draw[thick, densely dashed, blue!20] (2,0) to[out=125, in=-70] (0.6,3);
\draw[thick, densely dashed, blue] (0.6,3) to[out=110, in=-72] (0.43,3.5) ;
\draw[thick, densely dashed, blue] (1.985,0) -- (1.735,1);
\draw[thick, densely dashed, blue] (1.735,1) -- (0.6,3);;
\draw[thick, densely dashed, blue] (1.985,0) to[bend right=10] (1.785,-0.6);
\draw[thick, red!20] (2,0) to[out=88, in=-80] (1.9,3);
\draw[thick, red] (1.9,3) to[out=100, in=-78] (1.8,3.5);
\draw[thick, red] (2.015,0) -- (1.765,1);
\draw[thick, red] (1.765,1) -- (1.9,3);
\draw[thick, red] (2.015,0) to[bend right=10] (1.815,-0.6);
\draw[very thick, dotted, ForestGreen] (2,0) to[bend right=5] (3,3.5);
\draw[very thick, dotted, ForestGreen] (2,0) to[bend left=10] (2.6,-0.6);
\draw (2,0) node[left=6pt, below=-2pt]{$q$};
\draw (2pt,1) -- (-2pt, 1) node[left]{$t + \lambda\sigma$};
\draw (2pt,3) -- (-2pt, 3) node[left]{$t + \tau$};
\end{tikzpicture}
\caption{\label{fig:testtrajectory}To the left, a piece of the trajectory $\ZZ$, and to the right, the competitor $\ZZ^{\sigma,\lambda}$ that we describe in the proof.}
\end{figure}
(See Figure \ref{fig:testtrajectory} for an illustration of this competitor.) We will get a below bound on $v_2 - v_1$ by comparing the value of $\Lambda''$ on $\ZZ$ and $\ZZ^{\sigma, \lambda}$, and by differentiating the corresponding inequality first with respect to $\sigma$ at $\sigma=0$ (we zoom so that the particles of $\ZZ$ only travel along straight lines), and then with respect to $\lambda$ at $\lambda = 0$ (we compute the first variation of the action when we let the particles stick together).

\noindent \underline{Step 2}: A below bound on $v_2-v_1$.

The partitions $\pi(Z_s^{\sigma,\lambda})$ and $\pi(Z_s)$ coincide at all time except between time $t$ and $t + \lambda \sigma$, when $\pi(Z_s)$ is a strict refinement of $\pi(Z^{\sigma,\lambda}_s)$. Hence, calling:
\begin{equation*}
	\delta = \delta(N,A) := \min \big\{ h(\pi) - h(\pi') \, | \, (\pi, \pi') \mbox{ ordered partition of }\llbracket 1, N \rrbracket, \, \pi \mbox{ strict refinement of }\pi'\big\}>0,
\end{equation*}
we have for all $s \in (t,t + \lambda \sigma)$:
\begin{equation}
\label{eq:diff_h_delta}
h(\pi(Z_s^{\lambda,\sigma}))  + \delta \leq h(\pi(Z_s)).
\end{equation}
As $\ZZ^{\sigma}$ coincide with $\ZZ$ for times outside $(t, t + \sigma)$ and for coordinates that are not in $C_1 \cup C_2$, by definition~\eqref{eq:decomposed_functional_dim1} of $\Lambda''$, we have:
\begin{align}
\notag \Lambda''(\ZZ^{\sigma, \lambda}) - \Lambda''(\ZZ) &= \sum_{i \in C_1\cup C_2} \int_t^{t + \sigma}\Big\{ |\dot z_i^{ \sigma,\lambda}(s)|^2 + |z_i^{\sigma, \lambda}(s) - a_i|^2 -|\dot z_i(s)|^2 - |z_i(s) - a_i|^2\Big\}\D s \\
\notag &\hspace{2cm} + \int_t^{t + \lambda \sigma}\{ h(\pi(Z^{\sigma, \lambda}_s)) - h(\pi(Z_s)) \}\D s \\
\label{eq:estim_diff}&\leq \sum_{i \in C_1\cup C_2} \int_t^{t + \sigma}\Big\{ |\dot z_i^{\tau, \sigma}(s)|^2 -|\dot z_i(s)|^2 \Big\}\D s - \delta \lambda \sigma + \underset{\sigma \to 0}{o}(\sigma),
\end{align}
where to obtain the second line, we used \eqref{eq:diff_h_delta} and the fact that between times $t$ and $t + \sigma$, both $z_i$ and $z_i^{\sigma, \lambda}$ remain at a distance of order $\sigma$ of $q$.

Let us consider $i \in C_j$ for $j = 1,2$. One the one hand, as $z_i$ admits $v_j$ as a right derivative at time $t$, we have:
\begin{equation}
\label{eq:estim_dotzi}
\int_t^{t + \sigma}|\dot z_i(s)|^2 \D s = v_j^2 \sigma +  \underset{\sigma \to 0}{o}(\sigma). 
\end{equation}
On the other hand, we can compute explicitly:
\begin{align}
\notag \int_{t}^{t+ \sigma} |\dot z_i^{\sigma, \lambda}(s)|^2 \D s &= \lambda p^2 \sigma + (1-\lambda)\sigma \left( \frac{ z_i(t + \sigma) -  (q + \lambda p \sigma)}{(1-\lambda) \tau}\right)^2 \\
\notag &= \lambda p^2 \sigma + \frac{1}{(1-\lambda)\sigma} \left( q + v_j \sigma + \underset{\sigma \to 0}{o}(\sigma) - q - \lambda p \sigma \right)^2\\
\label{eq:esti_dotzsigma}&= \lambda p^2 \sigma +  \left( v_j  - \lambda p \right)^2 \frac{\sigma}{1-\lambda} + \underset{\sigma \to 0}{o}(\sigma).
\end{align}
By plugging \eqref{eq:estim_dotzi} and \eqref{eq:esti_dotzsigma} in \eqref{eq:estim_diff} and by using the definition \eqref{eq:def_kvp} of $k_1$, $k_2$ and $p$, we get:
\begin{align*}
\Lambda''(\ZZ^{\sigma, \lambda}) - \Lambda''(\ZZ) &\leq \Big\{ (k_1 + k_2)\lambda p^2 + \frac{k_1(v_1 - \lambda p)^2 + k_2(v_2 - \lambda p)^2}{1 - \lambda} - k_1 v_1^2 - k_2 v_2^2 - \delta \lambda \Big\}\sigma + \underset{\sigma \to 0}{o}(\sigma) \\
&=\Big\{  (k_1 + k_2) p^2 + k_1 v_1^2 + k_2 v_2^2 - 2 p\big( k_1 v_1 + k_2 v_2 \big) - \delta (1-\lambda)  \Big\}\frac{\lambda}{1-\lambda} \sigma + \underset{\sigma \to 0}{o}(\sigma)\\
&=\Big\{ k_1 v_1^2 + k_2 v_2^2 - \frac{(k_1v_1 + k_2 v_2)^2}{k_1 + k_2} - \delta (1-\lambda)  \Big\}\frac{\lambda}{1-\lambda} \sigma + \underset{\sigma \to 0}{o}(\sigma)\\
&=\Big\{\frac{k_1 k_2}{k_1 + k_2} (v_2 - v_1)^2 - \delta (1-\lambda)  \Big\}\frac{\lambda}{1-\lambda} \sigma + \underset{\sigma \to 0}{o}(\sigma).
\end{align*}
By minimality of $\Lambda''(\ZZ)$, this quantity must be nonnegative. If we divide it by $\lambda \sigma$, and if we let $\sigma$ and then $\lambda$ go to zero, we end-up with:
\begin{equation}
\label{eq:estim_v2-v1}
\frac{k_1 k_2}{k_1 + k_2} (v_2 - v_1)^2 \geq \delta.
\end{equation}

\noindent \underline{Step 3}: Conservation of momentum during an isolated shock and conclusion.

Because $(t,q,C)$ is isolated, it is easy to justify that we can replace $V$ by the vector $V^C$ whose $j$-th coordinate is $1$ if $j \in C$ and $0$ otherwise in the proof of Proposition~\ref{prop:conservation_impulsion}. Doing so, we obtain the "local" conservation of momentum:
\begin{equation*}
\frac{1}{\#C} \sum_{i \in C} \dot z_i(t-) = \frac{1}{\#C} \sum_{i \in C} \dot z_i(t+) =: \mathcal{P}^C(t).
\end{equation*}
by ordering of the particles, we have for $i = \min C$:
\begin{align*}
\dot z_i(t-) \geq \mathcal{P}(t) =\frac{1}{\#C} \sum_{i \in C} \dot z_i(t+) \geq \frac{k_1}{\#C} v_1 + \frac{\#C - k_1}{\#C} v_2.
\end{align*}
(Indeed, $j \in C \mapsto \dot z_j(t-)$ and $j \in C \mapsto \dot z_j(t+)$ are clearly non-increasing and non-decreasing respectively.) By recalling that $v_1 = \dot z_i(t+)$ and using \eqref{eq:estim_v2-v1}, we get:
\begin{equation*}
\dot z_i(t-) - \dot z_i(t+) \geq \frac{\#C - k_1}{\#C} (v_2 - v_1) \geq \frac{\#C - k_1}{\#C} \sqrt{\frac{k_1+k_2}{k_1k_2} \delta}.
\end{equation*}
The minimal right hand side's value is $\sqrt{\delta /(\#C^2 -\#C)}$, obtained for $k_1 = \#C - 1$ and $k_2=1$. Hence, we get the result by choosing $\alpha = \sqrt{\delta/(N^2 - N)}$.
\end{proof}
\subsection{Conclusion: proof of Theorem \ref{thm:traj_reg}}
\label{subsec:proof_regularity}
We are now ready to give the proof of Theorem \ref{thm:traj_reg}. We give ourselves $\ZZ$ a global minimizer of $\Lambda'$. Thanks to Proposition \ref{prop:ordered}, we can suppose that $\ZZ$ is ordered, and thanks to Proposition \ref{prop:decomposition}, we can consider $\Lambda''$ instead of $\Lambda'$. 

Because of Proposition \ref{prop:reg_isolated_shock}, it suffices to prove that there is a finite number of shocks. Indeed, in that case one can take for $0=t_0<t_1 \dots <t_p = T$ the moments of these shocks (and the endpoints of $[0,T]$). The smoothness of $\ZZ$ on each $[t_{i-1},t_i]$, $i=1, \dots, p$ follows directly from the Proposition \ref{prop:reg_isolated_shock}. Then $\pi(\ZZ)$ is constant on each $(t_{i-1}, t_i)$, $i=1, \dots, p$ because by Definition \ref{def:shock} of a shock, at each time of discontinuity of $\pi(\ZZ)$, there is at least one shock.

The set:
\begin{equation*}
\big\{ (t',q') \big| \, \exists C' \subset \llbracket1, N\rrbracket \mbox{ s.t. } (t',q',C') \mbox{ is a shock} \big\}
\end{equation*}
is easily seen to be compact. So if it is not finite, it admits at least one accumulation point. Otherwise stated, if there is an infinite number of shocks, then there is at least one shock which is not isolated. Let us consider such a shock $(t,q,C)$ with minimal number of particles involved, \emph{i.e.} with minimal $\#C$. The rest of the proof consists in showing that the existence of $(t,q,C)$ leads to a contradiction.

\noindent\underline{Step 1}: The velocities are bounded.

As $\ZZ$ is continuous on $[0,T]$, it is bounded. On the other hand, by definition, $h\leq |A|^2$. Now if $i \in \{1, \dots, N\}$ and $t \in [0,T]$ is such that $\ZZ$ is differentiable at $t$ (which is true for almost any $t$), recalling the definition \eqref{eq:def_energy} of $\mathcal{E}$:
\begin{equation*}
\dot z_i(t)^2 \leq |\dot Z_t|^2 \leq \mathcal{E} + |Z_t-A|^2 + h(\pi(Z_t)), 
\end{equation*}
which is bounded uniformly in $t$.

\noindent \underline{Step 2}: All the shocks in the neighbourhood of $(t,q)$ are isolated.

Let $U$ be a neighbourhood of $(t,q)$ in $[0,T]\times \R$ such that for all $s \in [0,T]$ and $i \in \{1, \dots, N\}$, $(s,z_i(s)) \in U$ implies $i \in C$. This is possible since $\ZZ$ is continuous and for all $j \notin C$, $z_i(t) \neq q$ by Definition~\ref{def:shock} of a shock. Let us consider $(t',q',C')$ a shock with $(t',q') \in U$. If $\#C' < \#C$, then $(t',q',C')$ is isolated by minimality of~$\#C$. If $\#C' = \#C$, then $C'=C$ by definition of $U$. But then it is easy to adapt the proof of Point 1. of Proposition~\ref{prop:full_degenerescence} to prove that $C \in \pi(Z_s)$, for all $s$ between $t$ and $t'$ and so there is no shock in $U$ between $t$ and $t'$. Hence there exists at most one such shock in $U$: either one before $t$ or one after $t$, but not both because else $(t,q,C)$ would contradict the third point of the definition of a shock. Up to reducing $U$, we can then exclude~$(t',q',C')$.

\noindent\underline{Step 3}: Conclusion using Proposition \ref{prop:reg_isolated_shock}.

As $(t,q,C)$ is not isolated, there is an infinite number of (isolated) shocks in $U$. Without loss of generality, we can assume that there is an infinite number of shocks in $U$ after time $t$. Call $i\in C$ the smallest index such that particle $i$ is involved in an infinite number of shocks in $U$ after time $t$. When $i \neq \min C$, up to reducing $U$ and by minimality of $i$, we can assume that no particle $j \in C$ with $j < i$ is involved in a shock in $U$ after time $t$.

As the shocks in $U$ involving $i$ after time $t$ are isolated (Step 2.), we can enumerate their times in decreasing order $(t_p)_{p \in \N}$. The boundedness of $\ZZ$ as well as~\eqref{eq:Newton_law} allows us to take $M$ an upper bound for $\ddot z_i$ between the times of shocks. For all $p \in \N$ and $s \in (t_{p+1}, t_p)$, taking $\alpha$ as in \eqref{eq:minimal_deviation}, we have:
\begin{align*}
\dot z_i(s) &= \dot z_i(t_0-) + \sum_{k=1}^p\left\{ \dot z_i(t_k-) - \dot z_i(t_k+) - \int_{t_{k-1}}^{t_k} \ddot z_i(\tau) \D \tau\right\} - \int_{t_p}^s \ddot z_i(\tau) \D \tau \\
&\geq \dot z_i(t_0-)+ p \alpha - M (t_0 - t),
\end{align*}
which contradicts Step 1. as soon as $p$ is sufficiently large. \qed

 \bibliography{bibliography}
 \bibliographystyle{plain}

 \end{document}